\documentclass{amsart}
\usepackage{subcaption}
\usepackage{graphicx}
\usepackage{amsmath}
\usepackage{amssymb}%
\usepackage{textcomp}
\usepackage{amsfonts}%
\usepackage{graphicx}
\usepackage[all]{xy}
\usepackage[mathscr]{euscript}
\usepackage{hyperref}
\newtheorem{theorem}{Theorem}[section]
\newtheorem{lemma}[theorem]{Lemma}
\theoremstyle{definition}
\newtheorem{definition}[theorem]{Definition}

\newtheorem{proposition}[theorem]{Proposition}
\newtheorem{corollary}[theorem]{Corollary}

\usepackage[all]{xy}
\usepackage{graphicx}
\theoremstyle{remark}
\newtheorem{remark}[theorem]{Remark}
\numberwithin{equation}{section}
\usepackage{graphicx}

\begin{document}
\title{\sc Non-uniqueness of the homotopy class of bounded curvature paths}
\author{Jos\'{e} Ayala}
\address{Instituto de Ciencias Exactas y Naturales ICEN, Universidad Arturo Prat, Iquique, Chile}
\email{jayalhoff@gmail.com}
\author{Hyam Rubinstein}
\address{Department of Mathematics and Statistics, University of Melbourne
              Parkville, VIC 3010 Australia}
\email{rubin@ms.unimelb.edu.au}
\maketitle

\begin{abstract}  A bounded curvature path is a continuously differentiable piecewise $C^2$ path with a bounded absolute curvature that connects two points in the tangent bundle of a surface. In this work, we analyze the homotopy classes of bounded curvature paths for points in the tangent bundle of the Euclidean plane. We show the existence of connected components of bounded curvature paths that do not correspond to those under (regular) homotopies obtaining the first results in the theory outside optimality.  An application to robotics is presented.
\end{abstract}

\section{Prelude}







Given a class of curves satisfying some constraints, understanding when there is a deformation connecting two curves in the class, where all the intermediate curves also are in the class, strongly relies on the defining conditions. When considering continuity any two plane curves (both closed or with different endpoints) are homotopic one into the other. Whitney in 1937 observed that under regular homotopies (homotopy through immersions) not always two planar closed curves lie in the same connected component \cite{whitney}. In fact, there are as many regular homotopy classes of plane curves as integers. When considering curves with different endpoints the concept of homotopies through immersions do not lead to results different from those obtained when only continuity is considered. Dubins in 1957 introduced the concept of bounded curvature path when characterizing bounded curvature paths of minimal length \cite{dubins 1}\footnote{Dubins proved that the length minimiser bounded curvature paths are paths being a concatenation of arcs of circles {\sc c} and a line segment {\sc s}. The {\sc csc}-{\sc ccc} paths.}. 

Let $(x,X),(y,Y)\in T{\mathbb R}^2$ be elements in the tangent bundle of the Euclidean plane. A planar bounded curvature path is a $C^1$ and piecewise $C^2$ path starting at $x$, finishing at $y$; with tangent vectors at these points $X$ and $Y$ respectively and having absolute curvature bounded by $\kappa=\frac{1}{r}>0$. Here $r$ is the minimum allowed radius of curvature. The piecewise $C^2$ property comes naturally due to the nature of the length minimisers \cite{dubins 1}.

A substantial part of the complexity of the theory of bounded curvature paths is described in the following observation. In general, length minimisers are used to establish a distance function between points in a manifold. This approach may not be considered for spaces of bounded curvature paths since in many cases the length variation between length minimisers of arbitrarily close endpoints or directions is discontinuous. Closely related is the fact that for most cases the length minimisers from $(x,X)$ to $(y,Y)$ and from $(y,Y)$ to $(x,X)$ have different length contradicting the symmetry property metrics satisfy. In addition, global length minimisers may not be unique. 

Spaces of bounded curvature paths may have up to six local minima (see Fig. \ref{exdubnonuniq}). In addition, some length minimisers in these paths spaces are not embedded. We first prove that a class of spaces of bounded curvature paths have isolated points being these of three types, see Theorem \ref{c}. Our main result, Theorem \ref{mainresultp1}, establishes the existence of additional homotopy classes of bounded curvature paths that do not correspond to those under (regular) homotopies. For a class of initial and final points $(x,X),(y,Y)\in T{\mathbb R}^2$ we prove the existence of a bounded plane region in which no embedded bounded curvature path defined on it is homotopic (without violating the curvature bound) to a path having a point in his  image not defined in such region. In other words, embedded bounded curvature paths get trapped in these plane regions.
 
 
We prove that embedded bounded curvature paths in these trapping regions are not homotopic while preserving the curvature bound to paths with self intersections, see Corollary \ref{cannotsing}. In order to prove Theorem \ref{mainresultp1} a core result called the S-Lemma is proven (see Lemma \ref{sthm}). This results relates the bound on curvature together with the turning map and its extremals. It is reasonable to conjecture that the space of embedded bounded curvature paths lying in these trapping regions forms an isotopy class. Here we present the first results on bounded curvature paths outside optimality. 



\begin{figure}[h]
	\centering
	\includegraphics[width=.8\textwidth,angle=0]{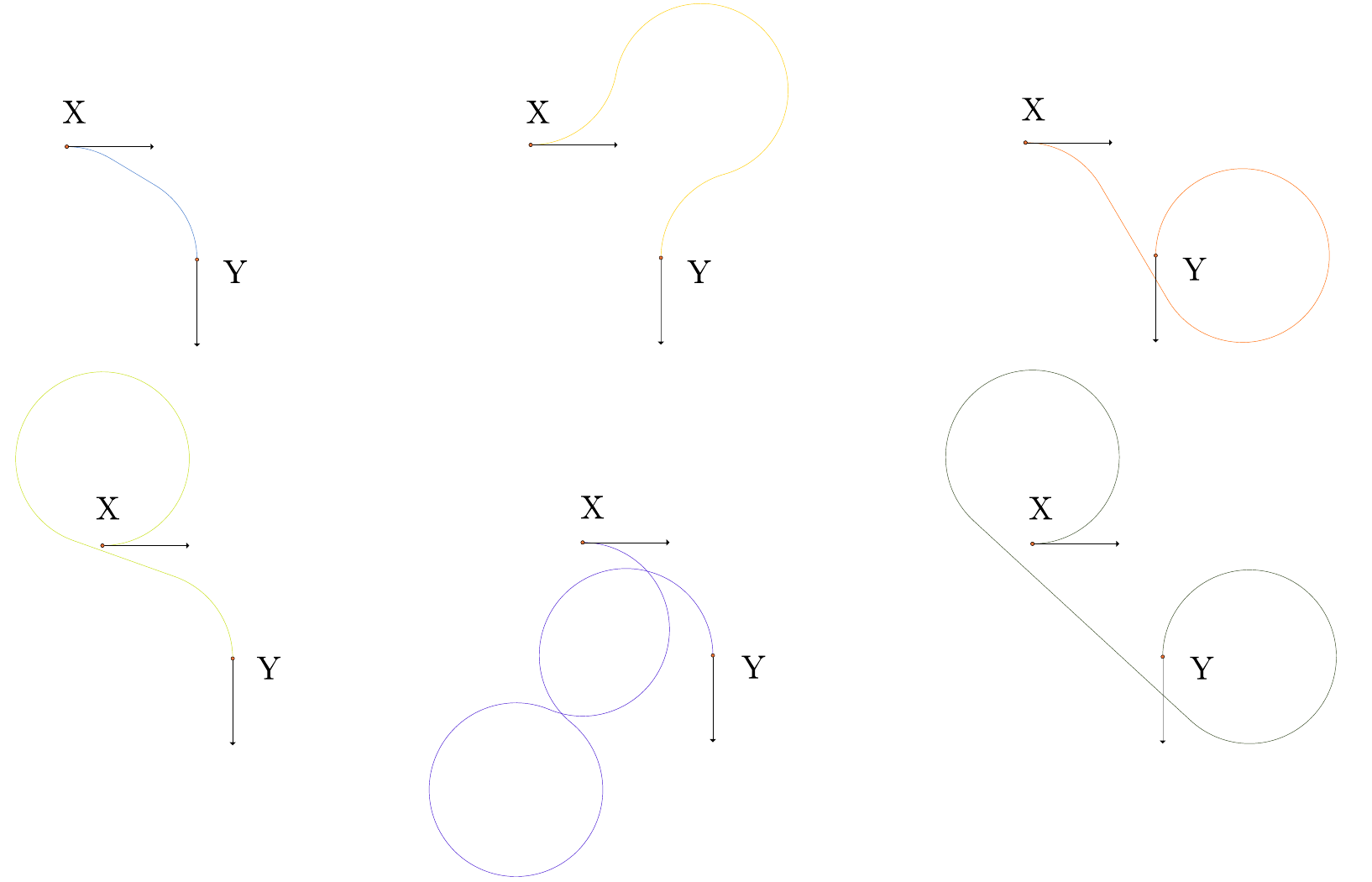}
\caption{Suppose $\mbox{\sc x}=(x,X),\mbox{\sc y}=(y,Y)\in T\mathbb R^2$ with $x=(0,0)$, $X=e^{2\pi i}\in T_x\mathbb R^2$, $y=(1.5,-1.3)$ and $Y=e^{-\frac{\pi}{2} i}\in T_y\mathbb R^2$. First note that all the paths shown are local minima of length (widely known as Dubins paths). Are these six paths in the same connected component? How many connected components are in the space of bounded curvature paths from $\mbox{\sc x}$ to $\mbox{\sc y}$? In Theorem \ref{mainresultp1} we establish that the first path is not homotopic (while preserving the curvature bound throughout the deformation) to any of the other five paths.}
	\label{exdubnonuniq}
\end{figure}

Bounded curvature paths have proven to be extremely useful in applications since a bound on the curvature is a turning circle constraint for the trajectory of wheeled robots (and Unmanned Aerial Vehicles \cite{lin1, owen1,tso1}) along paths. Applications can be found in computer science \cite{aga, bui, irina1, kirk, ny, rus}, control theory \cite{sig1, jur, mit, mon, mur, san, sig2, sus2} and engineering \cite{dot, brazil 1, chang, lin1, owen1,tso1}.

In the design of an underground mine (a network navigated by wheeled robots) the mine is considered as a 3-dimensional system of tunnels where the (directed) nodes correspond to: the surface portal, access points, and draw points. The links correspond to the centerlines of ramps and drives \cite{dot}. In addition, conditions of navigability as turning radius for vehicles, and gradient for ramps are required. Corollary \ref{noarblength} proves the optimality of the algorithm implemented in DOT, a software for constructing minimal length networks of systems of tunnels under the constraints previously described (see \cite{dot} and Fig. \ref{figmine}).

\section{Preliminaries}
Let us denote by $T{\mathbb R}^2$ the tangent bundle of ${\mathbb R}^2$. The elements in $T{\mathbb R}^2$ correspond to pairs $(x,X)$ sometimes denoted just by {\sc x}. As usual, the first coordinate corresponds to a point in ${\mathbb R}^2$ and the second to a tangent vector to ${\mathbb R}^2$ at $x$.

\begin{definition} \label{adm_pat} Given $(x,X),(y,Y) \in T{\mathbb R}^2$, we say that a path $\gamma: [0,s]\rightarrow {\mathbb R}^2$ connecting these points is a {\it bounded curvature path} if:
\end{definition}
 \begin{itemize}
\item $\gamma$ is $C^1$ and piecewise $C^2$.
\item $\gamma$ is parametrized by arc length (i.e $||\gamma'(t)||=1$ for all $t\in [0,s]$).
\item $\gamma(0)=x$,  $\gamma'(0)=X$;  $\gamma(s)=y$,  $\gamma'(s)=Y.$
\item $||\gamma''(t)||\leq \kappa$, for all $t\in [0,s]=:I$ when defined, $\kappa>0$ a constant.
\end{itemize}
Of course, $s$ is the arc-length of $\gamma$.

 The first condition means that a  bounded curvature path has continuous first derivative and piecewise continuous second derivative. We would like to point out that the minimal length elements in spaces of paths satisfying the last three items in Definition \ref{adm_pat} are in fact paths being $C^1$ and piecewise $C^2$. 
  
  For the third condition, without loss of generality, we can extend the domain of $\gamma$ to $(-\epsilon,s+\epsilon)$ for $\epsilon$ arbitrarily small. Sometimes we describe the third item as the endpoint condition. 
  
  The last condition means that bounded curvature paths have absolute curvature bounded above by a positive constant. Without loss of generality we consider $\kappa=1$ throughout this work, see Fig. \ref{figadj}. 
  
We consider the origin of our coordinate system as the base point $x$ with the standard basis $\{X=e_1,e_2\}$. Recall that $T{\mathbb R}^2$ is equipped with a natural projection $\pi : T{\mathbb R}^2 \rightarrow {\mathbb R}^2$. The fiber $\pi^{-1}(x)$ is ${\mathbb S}^1$ for all $x \in{\mathbb R}^2$. The space of endpoint conditions corresponds to a sphere bundle on ${\mathbb R}^2$ with fiber ${\mathbb S}^1$.

\begin{definition} \label{admsp} Given $\mbox{\sc x,y}\in T{\mathbb R}^2$ and a maximum curvature $\kappa>0$. The space of bounded curvature paths satisfying the given endpoint condition is denoted by $\Gamma(\mbox{\sc x,y})$. \end{definition}

 Throughout this work we consider $\Gamma(\mbox{\sc x,y})$ with the topology induced by the $C^1$ metric. It is important to note that properties (among many others) such as number of connected components and number of local minima in $\Gamma(\mbox{\sc x,y})$ depend on the chosen elements in $T{\mathbb R}^2$. 
 
\begin{definition}Let $\mbox{\sc C}_ l(\mbox{\sc x})$ be the unit circle tangent to $x$ and to the left of $X$. Analogous interpretations apply for $\mbox{\sc C}_ r(\mbox{\sc x})$, $\mbox{\sc C}_ l(\mbox{\sc y})$ and $\mbox{\sc C}_ r(\mbox{\sc y})$ (see Fig. \ref{figadj} left). These circles are called {\it adjacent circles}. Denote their centers with lowercase letters, so the center of $\mbox{\sc C}_ l(\mbox{\sc x})$ is denoted by $c_l(\mbox{\sc x})$. 
\end{definition}

\begin{figure}[h]
	\centering
	\includegraphics[width=1\textwidth,angle=0]{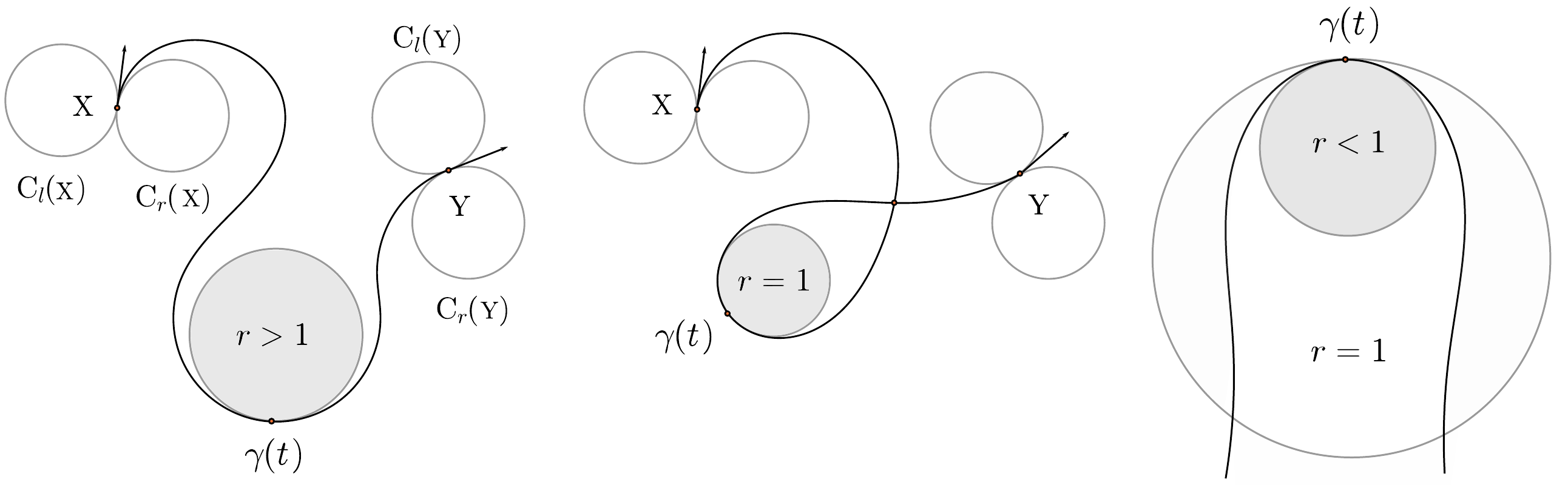}
\caption{Checking the bound on curvature at $\gamma(t)$. The first two paths have curvature bounded by $\kappa=1$. The third path has curvature greater to 1 at $\gamma(t)$. The bound on curvature is violated since $r<1$ at $\gamma(t)$ implies $\kappa>1$. }
	\label{figadj}
\end{figure}

We adopt the following convention: When a path is continuously deformed under parameter $p$, we reparametrize each of the deformed paths by its arc-length. Thus $\gamma: [0,s_p]\rightarrow {\mathbb R}^2$ describes a deformed path at parameter $p$, with $s_p$ corresponding to its arc-length.

\begin{definition}  \label{hom_adm} Given $\gamma,\eta \in \Gamma(\mbox{\sc x,y})$. A {\it bounded curvature homotopy}  between $\gamma: [0,s_0] \rightarrow \mathbb R^2$ and $\eta: [0,s_1] \rightarrow \mathbb R^2$ corresponds to a continuous one-parameter family of immersed paths $ {\mathcal H}_t: [0,1] \rightarrow \Gamma(\mbox{\sc x,y})$ such that:
\begin{itemize}
\item ${\mathcal H}_t(p): [0,s_p] \rightarrow \mathbb R^2$ for $t\in [0,s_p]$ is an element of $\Gamma(\mbox{\sc x,y})$ for all $p\in [0,1]$.
\item $ {\mathcal H}_t(0)=\gamma(t)$ for $t\in [0,s_0]$ and ${\mathcal H}_t(1)=\eta(t)$ for $t\in [0,s_1]$.
\end{itemize}
\end{definition}

The next remark summarizes well known facts about homotopy classes of paths on metric spaces. These facts are naturally adapted for elements in $\Gamma(\mbox{\sc x,y})$ with $\mbox{\sc x,y}\in T{\mathbb R^2}$. Recall that we are considering $\Gamma(\mbox{\sc x,y})$ with the topology induced by the $C^1$ metric.

\begin{remark}(\it On homotopy classes). Given $\mbox{\sc x,y}\in T{\mathbb R^2}$ then:
\end{remark}
\begin{itemize}
\item Two bounded curvature paths are {\it bounded-homotopic} if there exists a bounded curvature homotopy from one path to another. The previously described relation defined by $\sim$ is an equivalence relation.
\item A {\it homotopy class} in $\Gamma(\mbox{\sc x,y})$ corresponds to an equivalence class in $\Gamma(\mbox{\sc x,y})/\sim$.
\item A {\it homotopy class} is a maximal path connected set in $\Gamma(\mbox{\sc x,y})$.
\end{itemize}

\begin{definition}  A bounded curvature path is said to be {\it free} if is bounded-homotopic to a path of arbitrary large length.
\end{definition}

\section{Proximity of endpoints}\label{proximity}

Next we obtain four simple pairs of inequalities. These allow us to reduce the study of configurations of endpoints in $T\mathbb R^2$ to a finite number of cases (up to isometries). One of these conditions partially characterize the spaces of bounded curvature paths containing only embedded paths. 

 The following relations are obtained by analysing the configurations for the adjacent circles in the plane  (see Fig. \ref{figccproxcond2}).


 \begin{equation} d(c_l(\mbox{\sc x}),c_l(\mbox{\sc y}))\geq 4 \quad \mbox{and}\quad d(c_r(\mbox{\sc x}),c_r(\mbox{\sc y}))\geq4 \label{con_a}\tag{i}\end{equation}
 \begin{equation} d(c_l(\mbox{\sc x}),c_l(\mbox{\sc y}))< 4 \quad \mbox{and}\quad d(c_r(\mbox{\sc x}),c_r(\mbox{\sc y}))\geq 4 \label{con_b}\tag{ii} \end{equation}
  \begin{equation} d(c_l(\mbox{\sc x}),c_l(\mbox{\sc y}))\geq4 \quad \mbox{and}\quad d(c_r(\mbox{\sc x}),c_r(\mbox{\sc y}))< 4  \label{con_b'}\tag{iii} \end{equation}
   \begin{equation} d(c_l(\mbox{\sc x}),c_l(\mbox{\sc y}))< 4 \quad \mbox{and}\quad d(c_r(\mbox{\sc x}),c_r(\mbox{\sc y}))< 4 \label{con_c}\tag{iv} \end{equation}
   \vspace{.3cm}

Note that by definition $d(c_l(\mbox{\sc x}),c_r(\mbox{\sc x}))=2$ and $d(c_l(\mbox{\sc y}),c_r(\mbox{\sc y}))=2$.  In addition, it is easy to see that the first two paths in Fig. \ref{figadj} satisfiy (i). In Fig. \ref{figc} we emphasise that several configurations are possible under condition (iv).

\begin{figure}[h]
	\centering
	\includegraphics[width=.7\textwidth,angle=0]{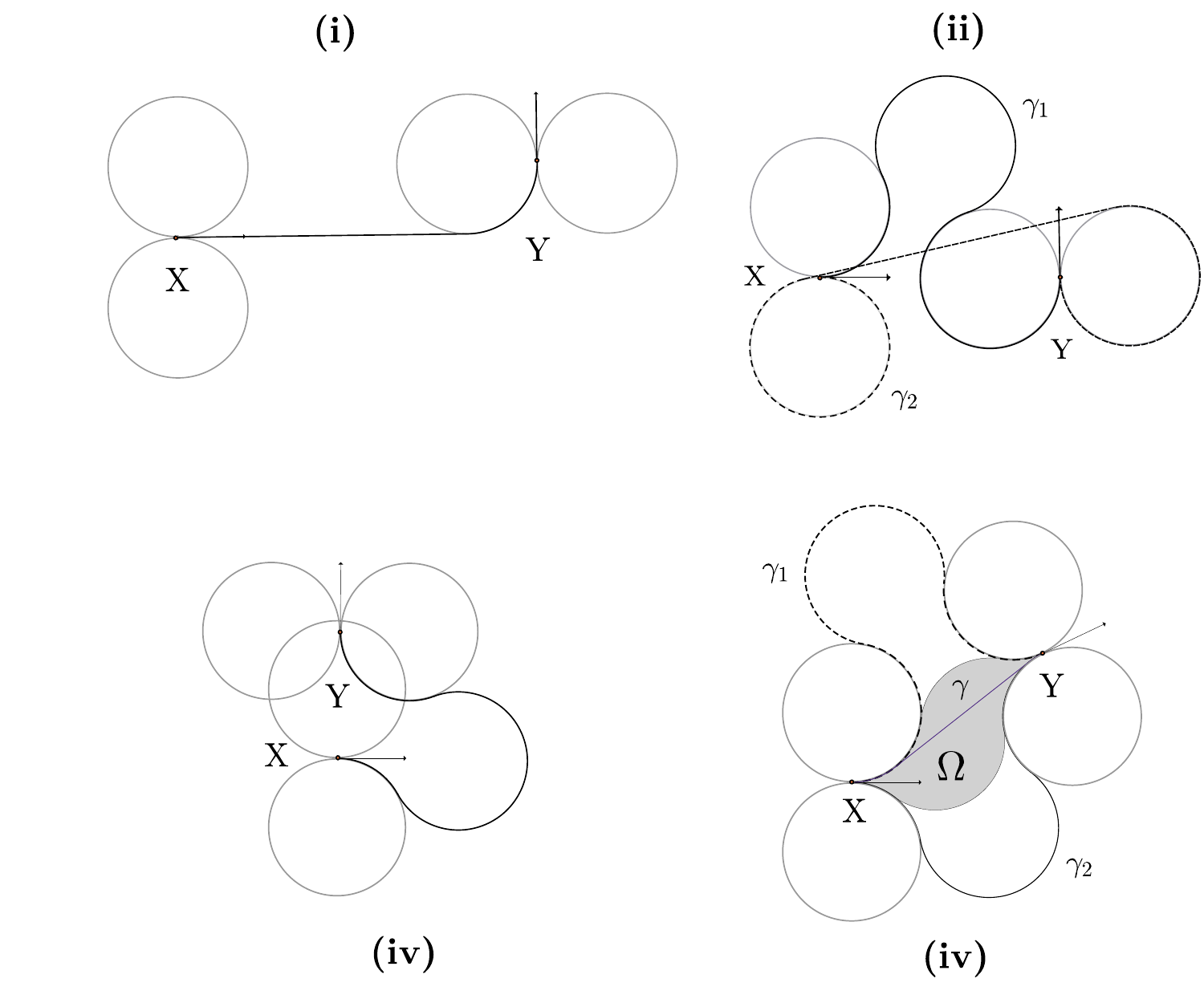}
\caption{Examples of endpoints satisfying (i), (ii) and (iv). Note that (ii) and (iii) are equivalent up to isometries. Upper right corner: Are $\gamma_1$ and $\gamma_2$ bounded-homotopic? Lower right corner: In Theorem \ref{mainresultp1} we prove that $\gamma_1$ and $\gamma_2$ are not bounded-homotopic to $\gamma$. All the paths shown are local minima of length. }
\label{figccproxcond2}
\end{figure}

\begin{figure}[h]
	\centering
	\includegraphics[width=.8\textwidth,angle=0]{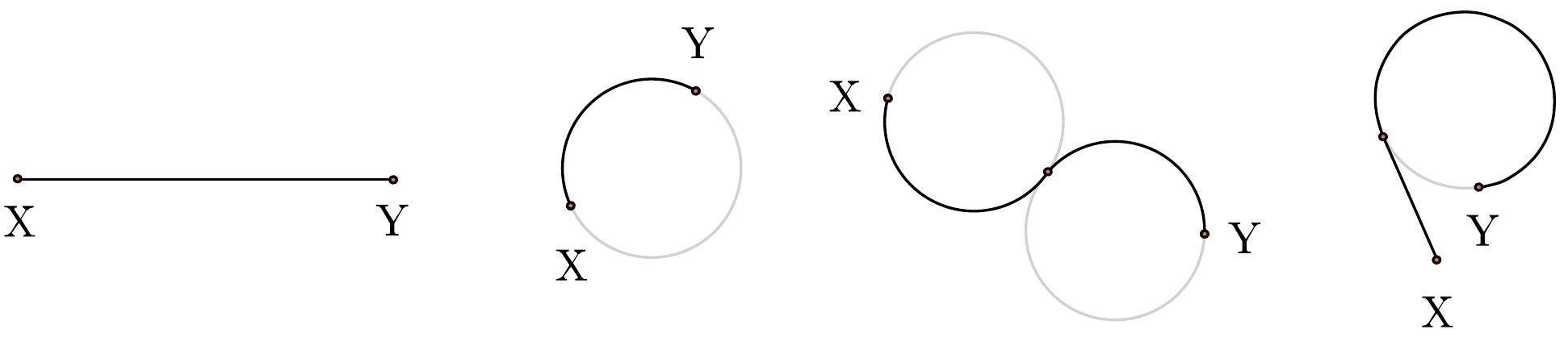}
\caption{Examples of bounded curvature paths whose adjacent circles (omitted in the illustration) satisfy condition (iv). In Theorem \ref{c} we show that the second and third paths are isolated points in their respective spaces $\Gamma(\mbox{\sc x,y})$. Note that the last path has parallel tangents and therefore is free (see Proposition \ref{partan}). For the sake of clarity sometimes we omit the initial and final tangent vectors.}
 \label{figc}
\end{figure}

\begin{figure}[h]
	\centering
	\includegraphics[width=.6\textwidth,angle=0]{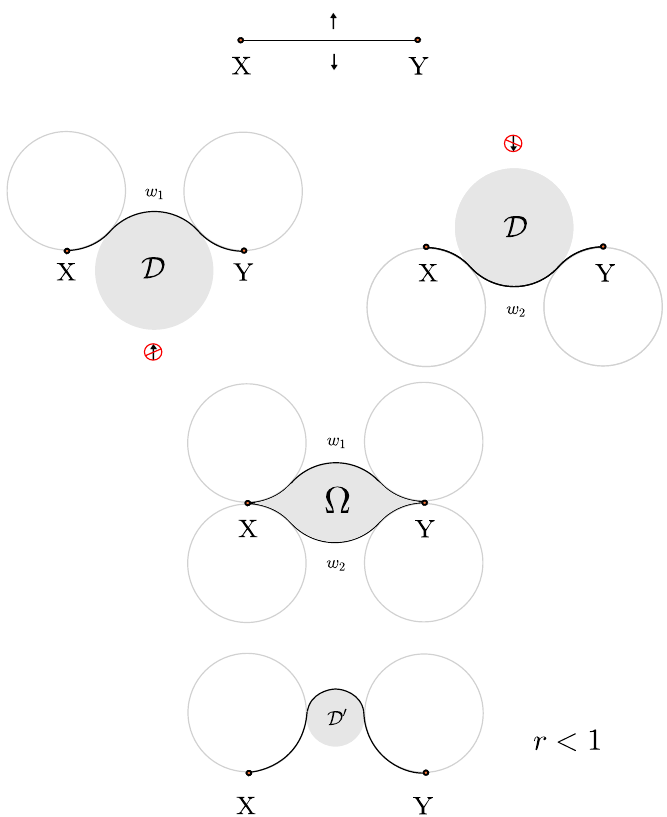}
\caption{Constructing regions $\Omega\subset \mathbb R^2$ as described in Remark \ref{regdefn}.}
 \label{figmoviereg}
\end{figure}

Next we describe a type of configuration satisfying (iv) of particular interest.

\begin{remark}{\bf (Constructing $\Omega\subset \mathbb R^2$)}. \label{regdefn} Consider $(x,X),(y,Y)\in T\mathbb R^2$ with $x=(0,0)$, $X=e^{2\pi i}\in T_x\mathbb R^2$, $y=(3,0)$ and $Y=e^{2\pi i}\in T_y\mathbb R^2$. Note that the endpoint condition satisfies (iv). Let $\gamma$ be the line segment joining $(x,X)$ with $(y,Y)$ (you may imagine that $\gamma$ is made out of rubber) see Fig. \ref{figmoviereg} top. By sliding a unit disk $\mathcal D$ along $\gamma$ up we obtain a path $w_1$. Note that $\mathcal D$ gets stuck in between $\mbox{\sc C}_ l(\mbox{\sc x})$ and $\mbox{\sc C}_ l(\mbox{\sc y})$. This happen since $d(c_l(\mbox{\sc x}),c_l(\mbox{\sc y}))< 4$ and because the initial and final points and directions are fixed. Note that the centers of the three circles involved are the vertices of an isosceles triangle with two sides of length 2 and the other one of length $d(c_l(\mbox{\sc x}),c_l(\mbox{\sc y}))$. If the centers of the three circles in question are collinear then $D'$ has radius less to 1 implying that the continuous deformation of $\gamma$ (see Fig. \ref{figmoviereg} bottom) violates the curvature bound. Analogously, a curve $w_2$ is obtained by a similar process, this time $\mathcal D$ gets stuck in between $\mbox{\sc C}_ r(\mbox{\sc x})$ and $\mbox{\sc C}_ r(\mbox{\sc y})$.

By concatenating $w_1$ and $w_2$ we obtain a simple closed plane curve. And, by the Jordan curve theorem, the complement of such a curve consists of exactly two connected components in $\mathbb R^2$. One is bounded while the other is unbounded.
\end{remark}

\begin{figure}[h]
	\centering
	\includegraphics[width=1\textwidth,angle=0]{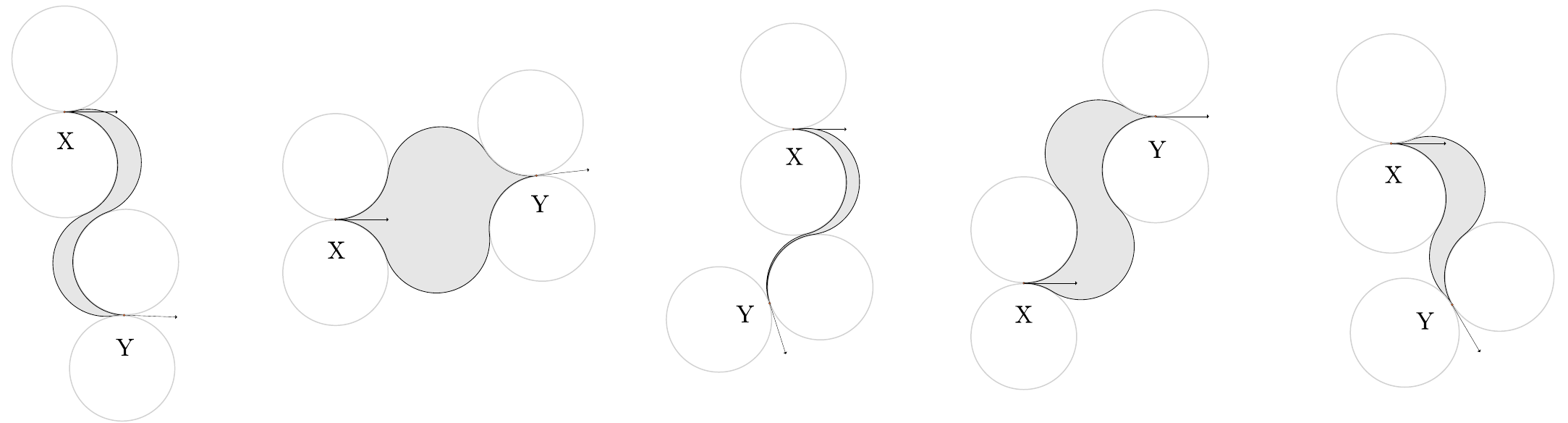}
\caption{In grey we illustrate some regions $\Omega\subset \mathbb R^2$ as described in Remark \ref{regdefn}. Note that the shape of the regions is determined by the initial and final points in $T\mathbb R^2$. Can these spaces of paths be parametrized so that we can describe exactly for what points in $\mbox{\sc x,y}\in T{\mathbb R}^2$ these regions $\Omega$ are obtained?}
	\label{figregex}
\end{figure}

\begin{definition} Suppose that the construction in Remark \ref{regdefn} can be performed.  Let $\Omega\subset \mathbb R^2$ be the closure of the bounded component enclosed by $w_1$ and $w_2$ in Remark \ref{regdefn}. In this case we say that $\mbox{\sc x,y}\in T{\mathbb R}^2$ carries a region $\Omega$.  A path $\gamma: I \to \mathbb R^2$ is said to be {\it in} $\Omega$ if $\gamma(t)\in \Omega$ for all $t\in I$. Otherwise $\gamma$ is said to be {\it not in} $\Omega$.  The boundary of $\Omega$ is denoted by $\partial \Omega$.
\end{definition}

\begin{remark} \label{condcconf}If $\mbox{\sc x,y}\in T{\mathbb R}^2$ satisfies (iv) we have three mutually exclusively cases:
\begin{itemize}
\item $\mbox{\sc x,y}\in T\mathbb R^2$ carries a region $\Omega\subset \mathbb R^2$ (see Fig. \ref{figregex}).
\item $\mbox{\sc x,y}\in T{\mathbb R}^2$ is the endpoint condition of a path consisting of a single arc of a unit circle of length less than $\pi$ or, $\mbox{\sc x,y}\in T{\mathbb R}^2$ is the endpoint condition of a path consisting of a concatenation of two arcs of unit circles each of length less than $\pi$ (see the second and third illustrations in Fig. \ref{figc}).
\item $\mbox{\sc x,y}\in T{\mathbb R}^2$ is the endpoint condition of a free path (see Fig. \ref{figc} right).
\end{itemize}
\end{remark}

\begin{definition}\label{condcnoreg} \hfill 
\begin{itemize}
\item If $\mbox{\sc x,y}\in T\mathbb R^2$ satisfies (i) we say that $\Gamma({\mbox{\sc x,y}})$ satisfies condition {\sc A}.
\item If $\mbox{\sc x,y}\in T\mathbb R^2$ satisfies (ii) or (iii) we say that $\Gamma({\mbox{\sc x,y}})$ satisfies condition {\sc B}.
\item If $\mbox{\sc x,y}\in T\mathbb R^2$ satisfies (iv) and $\Gamma({\mbox{\sc x,y}})$ contains a path that has as a subpath being:
\begin{itemize}
\item an arc of circle of length greater than or equal to $\pi $, or
\item a line segment of length greater than or equal to $4$.
\end{itemize}
we say that $\Gamma({\mbox{\sc x,y}})$ satisfies condition {\sc C}.
\end{itemize}
\end{definition}

\begin{definition}\label{defd} Suppose that $\mbox{\sc x,y}\in T\mathbb R^2$ satisfies (iv). We say that $\Gamma({\mbox{\sc x,y}})$ satisfies condition {\sc D} if:

\begin{itemize}
\item $\mbox{\sc x,y}\in T\mathbb R^2$ carries a region $\Omega\subset \mathbb R^2$, or
\item $\Gamma(\mbox{\sc x,y})$ contains a path consisting of an arc of a unit circle of length less than $\pi $, or
\item $\Gamma(\mbox{\sc x,y})$ contains a path being a concatenation of two arcs of unit circle of length less than $\pi $ each.
\end{itemize}
\end{definition}

It is not hard to see that the three items in Definition \ref{defd} are mutually exclusive. 



\section{Some remarks about $\Omega$} \label{domaincp}

Recall that in Remark \ref{regdefn} we obtained the curves $w_1$ and $w_2$ being each a bounded curvature path consisting of a concatenation of three arcs of unit circle. The four inflection points in $w_1$ and $w_2$ shown in Fig.~\ref{fignot} are denoted by $i$ and are indexed as elements in ${\mathbb R}^2$ with $|| i_k ||\leq ||i_{k+1}||$ for $k=1,2,3$. 

Set $\Theta_k$ as the smallest circular arc on the appropriate adjacent circle (whose length we denote by $\theta_k$) starting from $x$ or $y$ and finishing at $i_k$ with $1\leq k \leq 4$ (see Fig. \ref{fignot}). Denote by $\ell_1$ and $\ell_2$ the lines joining the first two and the last two indexed inflection points respectively. 

\begin{figure}[h]
	\centering
	\includegraphics[width=.7\textwidth,angle=0]{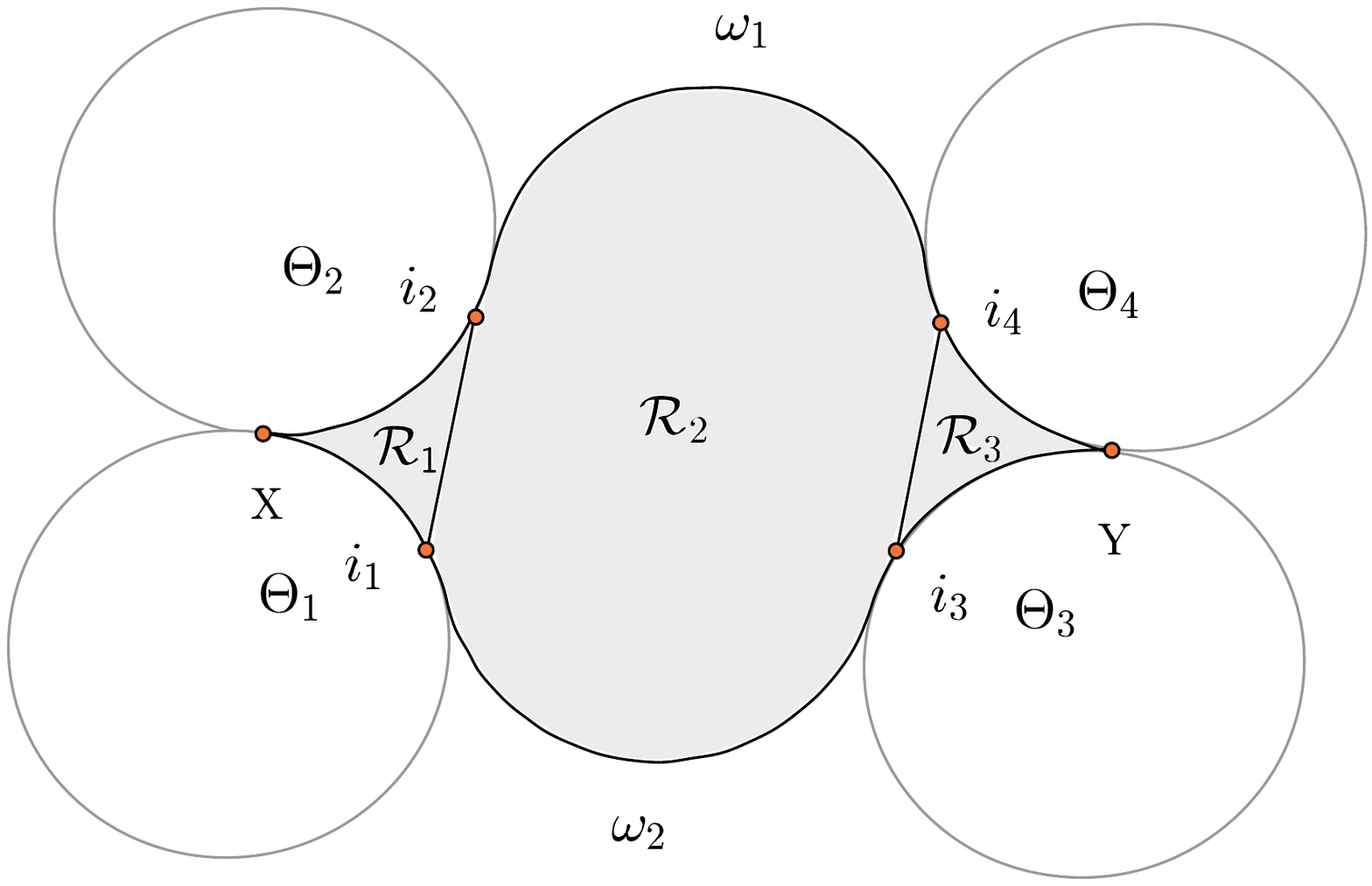}
	\caption{Notation for a generic $\Omega\subset \mathbb R^2$.}
\label{fignot}
\end{figure}

We subdivide $\Omega\subset \mathbb R^2$ into three subregions as follows:

\begin{itemize}
\item Let ${\mathcal R}_1$ be the closed portion of $\Omega$ which is to the left of $\ell_1$.
\item Let ${\mathcal R}_2$ be the closed portion of $\Omega$ which is between $\ell_1$ and $\ell_2$.
\item Let ${\mathcal R}_3$ be the closed portion of $\Omega$ which is to the right of $\ell_2$.
\end{itemize}

If $\theta_2>\frac{\pi}{2}$, then $\ell_1$ crosses twice the adjacent arc $\Theta_2$. In this case we replace $\ell_1$ by the line joining $i_1$ and a point $z$ in the adjacent arc $\Theta_2$ such that $\ell_1$ is tangent to $\Theta_2$ at $z$. The same idea is applied if $\theta_4>\frac{\pi}{2}$.

Recall that under condition {\rm D} we have that:
$$d(c_l(\mbox{\sc x}),c_l(\mbox{\sc y}))<4 \qquad \mbox{and}  \qquad d(c_r(\mbox{\sc x}),c_r(\mbox{\sc y}))<4$$

In particular we have that the center of the upper arc in $\partial {\mathcal R}_2$ is located below the line segment  with endpoints $c_l(\mbox{\sc x})$ and $c_l(\mbox{\sc y})$, and the center of the lower arc in $\partial {\mathcal R}_2$ is located above the line segment  with endpoints $c_r(\mbox{\sc x})$ and $c_r(\mbox{\sc y})$. Observe that otherwise the middle circular arcs in $\partial {\mathcal R}_2$ would have length greater or equal than $\pi$ which contradicts the formation of $\Omega$. 

\begin{lemma} \label{pi/2} The length of the circular arcs in $\partial \Omega$ satisfies:
\begin{itemize}
\item if $\theta_2\geq \frac{\pi}{2}$, then $\theta_1<\frac{\pi}{2}$.
 \item if $\theta_4\geq\frac{\pi}{2}$, then $\theta_3<\frac{\pi}{2}$.
 \end{itemize}
\end{lemma}

\begin{proof} Suppose that $\theta_1\geq \frac{\pi}{2}$, with $\theta_2\geq\frac{\pi}{2}$. Consider the following subsets of ${\mathbb R}^2$ $\mathcal{S}^+=\{(u,v)\in{\mathbb R}^2\,|\, u\geq 0\,\,v>1\}$ and $ \mathcal{S}^-=\{(u,v)\in{\mathbb R}^2\,|\, u\geq 0\,\,v<-1\}$. Without loss of generality suppose that $\theta_1$ lies in  $\mbox{\sc C}_l(\mbox{\sc x})$, then consider the line $l_1$ starting at $c_l(\mbox{\sc x})$ and passing through the inflection point $i_1$. Note that the center of the boundary arc $\partial{\mathcal R}_2$ adjacent to  $\mbox{\sc C}_l(\mbox{\sc x})$ must lie on $l_1$. Since the length of the boundary arc $\partial{\mathcal R}_2$ adjacent to  $\mbox{\sc C}_l(\mbox{\sc x})$ is less than or equal to $\pi$ this implies that $c_l(\mbox{\sc y})$ lies in the interior of the upper half space of $l_1$. Similarly, if $l_2$ is the line passing through $c_r(\mbox{\sc x})$ and $i_2$, then $c_r(\mbox{\sc y})$ lies in the interior of the lower half space of $l_2$. Observe that $c_l(\mbox{\sc y})$ lies in the interior of ${\mathcal S}^+$ and that $c_r(\mbox{\sc y})$ lies in the interior of ${\mathcal S}^-$ implying that  $d(c_l(\mbox{\sc y}),c_r(\mbox{\sc y})>2$ which leads to a contradiction. An analogous approach proves the second statement.
\end{proof}


\begin{corollary} \rm  $\partial \Omega$ contains at most two arcs $\Theta_k$ such that, $\theta_k >\frac{\pi}{2}$.
\end{corollary}
\begin{proof} Suppose that $\theta_k>\frac{\pi}{2}$, for three indices. So, we have $4$ possible arrangements satisfying the hypothesis, these are $\{\theta_1,\theta_2,\theta_3\}$,  $\{\theta_1,\theta_2,\theta_4\}$, $\{\theta_1,\theta_3,\theta_4\}$, and $\{\theta_2,\theta_3,\theta_4\}$. Observe that for $\theta_1,\theta_2>\frac{\pi}{2}$ Lemma \ref{pi/2} immediately implies that $d(c_l(\mbox{\sc y}),c_r(\mbox{\sc y}))>2$. Similarly if $\theta_3,\theta_4>\frac{\pi}{2}$ by the same argument implies that $d(c_l(\mbox{\sc x}),c_r(\mbox{\sc x}))>2$. Since $\theta_1,\theta_2$ or $\theta_3,\theta_4$ are contained in all the possible arrangements of three angles the result follows.
\end{proof}

\begin{remark}Observe that the combinations $\theta_1,\theta_3>\frac{\pi}{2}$;  $\theta_1,\theta_4>\frac{\pi}{2}$; $\theta_2,\theta_3>\frac{\pi}{2}$; $\theta_2,\theta_4>\frac{\pi}{2}$ are the only possible configurations for the lengths of the adjacent arcs when $\Omega$ contains two arcs $\Theta_k$ with lengths greater than $\frac{\pi}{2}$.
\end{remark}

\section{On the diameter of $\Omega$}\label{secdiam}

In Section \ref{sectsthm} we prove a core result characterizing embedded bounded curvature paths in $\Omega\subset \mathbb R^2$. The S-Lemma (Lemma \ref{sthm}) relates the diameter of $\Omega$, the turning map, together with the existence of maximal inflection points. In this section we give an upper bound for the diameter of $\Omega$ for all $\mbox{\sc x,y}\in T\mathbb R^2$ satisfying condition D.


\begin{proposition}\label{ntc} The two unit disks defined by extending the two middle arcs in $\partial{\mathcal R}_2$ have intersecting interiors.
\end{proposition}

\begin{proof} Recall that $\Omega\subset \mathbb R^2$ is obtained when simultaneously,
$$d(c_l(\mbox{\sc x}),c_l(\mbox{\sc y}))<4 \quad \mbox{and} \quad d(c_r(\mbox{\sc x}),c_r(\mbox{\sc y}))<4.$$
Suppose the two unit radius disks ${\mathcal D}_1,{\mathcal D}_2$ defined by the middle arcs have disjoint interiors. The distance of their centers $c_1,c_2$ satisfies $d(c_1,c_2)\geq2$. Let $\mathcal Q$ be the quadrilateral with vertices $c_l(\mbox{\sc x}), c_r(\mbox{\sc x}),c_l(\mbox{\sc y})$, and $c_r(\mbox{\sc y})$. Observe that $d(c_l(\mbox{\sc x}),c_l(\mbox{\sc y}))<4$ and $d(c_r(\mbox{\sc x}),c_r(\mbox{\sc y}))<4$ implies that $c_1$ and $c_2$ are points in the interior of ${\mathcal Q}$. By condition D, the middle arcs in $\partial {\mathcal R}_2$ have length less than $\pi$.  Observe that,
$$2=d(c_l(\mbox{\sc x}),c_1)=d(c_l(\mbox{\sc y}),c_1)=d(c_r(\mbox{\sc x}),c_2)=d(c_r(\mbox{\sc y}),c_2)=$$
$$d(c_1,c_2)=d(c_l(\mbox{\sc x}),c_l(\mbox{\sc y}))=d(c_r(\mbox{\sc x}),c_r(\mbox{\sc y}))$$
  It is easy to see that $c_1$ and $c_2$ are in ${\mathcal Q}$ and that $d(c_1,c_2)\geq 2$ is impossible.
 \end{proof}

\begin{corollary}\label{distR2} If $w$ and $z$ belong to opposite components of  $\partial {\mathcal R}_2$ then, $$d(w,z)<4.$$
\end{corollary}
\begin{proof}  After extending the middle arcs in $\partial{\mathcal R}_2$ to circles we see that Proposition \ref{ntc} ensures that these circles overlap. Therefore the distance of any two points in opposite arcs of $\partial{\mathcal R}_2$ is bounded above by $4$.
\end{proof}


\begin{remark} \label{minmax} Given two disjoint circles ${C}_1$ and ${C}_2$ in the plane. The line connecting their centers intersects the circles at four points $\{w, x, y, z\}$; we denote by ${XY}$ the line segment joining the points $x\in C_1$ and $y\in C_2$, which has interior disjoint from the circles, and by ${WZ}$ the long segment joining the points $w\in C_1$ and $z\in C_2$. The following claims are valid since under the presented conditions there are only four critical points for the distance function between points on the circles, one maximum, one minimum and two saddle points:

\begin{itemize}
\item The maximum of the distances between points in  $C_1$ and $C_2$ is given by the endpoints of the segment $WZ$. That is,
\[ \ \operatorname*{max}_{a\in C_1\,\, b\in C_2} d(a,b)=d(z,w).\]

 \item The minimum of the distances of points in  $C_1$ and $C_2$ is given by the endpoints of the segment $XY$. That is,
\[ \ \operatorname*{min}_{a\in C_1\,\, b\in C_2} d(a,b)=d(x,y).\]
\end{itemize}
We leave the details to the reader.
\end{remark}

\begin{theorem}\label{diaml4} The diameter of $\Omega\subset \mathbb R^2$ is strictly bounded by 4.
\end{theorem}
\begin{proof} Since the region $\Omega$ is compact and the Euclidean distance is continuous we have that $d:\Omega \times \Omega \rightarrow {\mathbb R}^{\geq 0}$ attains a maximum. Naturally, the maximum of $d$ must be achieved in $\partial \Omega$. In addition, note that $d(x,y)$ may or may not be a local maximum whenever $\Omega$ exists, see Fig. \ref{figregex} for examples of both cases.
Corollary \ref{distR2} establishes an upper bound for the distances between the opposite arcs in $\partial{\mathcal R}_2$ and, by the first item in Remark \ref{minmax} extended to overlapping circles such a value is a local maximum only if such arcs intersect the line joining the respective centers of the overlapping circles inside $\Omega$. The second item in Remark \ref{minmax} establishes that the distances between opposite arcs in $\partial{\mathcal R}_1$ and $\partial{\mathcal R}_3$ are at most candidates to be saddles or local mimima. We have that: 
 $$\mbox{\rm diam}(\Omega)=\max \{d(x,y),d(w,z)\},$$
 concluding that $\mbox{\rm diam}(\Omega)<4$  as desired.
\end{proof}


\section{Existence of isolated points}

The proof of Theorem \ref{c}, the main result in this section, requires Lemma 3.1 in \cite{papere}. This lemma states that a bounded curvature path with the initial and final points being distant apart less than $2$ cannot be defined above a unit radius circle while exclusively lying in an open band of width $2$ and arbitrary height, see Fig. \ref{figband}. This lemma is independent of wether or not we fix the initial and final vectors. We encourage the reader to refer to \cite{papere} for details.

\begin{lemma}\label{r1} (Lemma 3.1 in \cite{papere}). A bounded curvature path $\gamma: I \rightarrow {\mathcal B}$ where,
$${\mathcal B}=\{(x,y)\in{\mathbb R}^2\,|\, -1<x<1\,\,,\,\,y\geq 0 \}$$
cannot satisfy both:
\begin{itemize}
\item $\gamma(0),\gamma(s)$ are points on the $x$-axis.
\item If $C$ is a unit circle with center on the negative $y$-axis, and $\gamma(0),\gamma(s)\in C$, then some point in $\mbox{\rm Im}(\gamma)$ lies above $C$.
\end{itemize}
\end{lemma}

The next definition will be of relevance when concluding the proof of the S-Lemma (see Lemma \ref{sthm}).

\begin{definition}\label{defncross} Let $L_1$ and $L_2$ be the lines $x=-1$ and $x=1$ respectively.  A line joining two points in $\gamma$ distant apart at least $2$ one to the left of $L_1$ and the other to the right of $L_2$ is called a cross section (see dashed trace in Fig. \ref{figband} right).
\end{definition}

\begin{definition} A plane curve $\gamma$ has {\it parallel tangents} if there exist $t_1,t_2 \in I$, with $t_1<t_2$, such that $\gamma'(t_1)$ and $\gamma'(t_2)$ are parallel and pointing in opposite directions (see Fig. \ref{figpartan}).
\end{definition}

\begin{figure}[h]
	\centering
	\includegraphics[width=1\textwidth,angle=0]{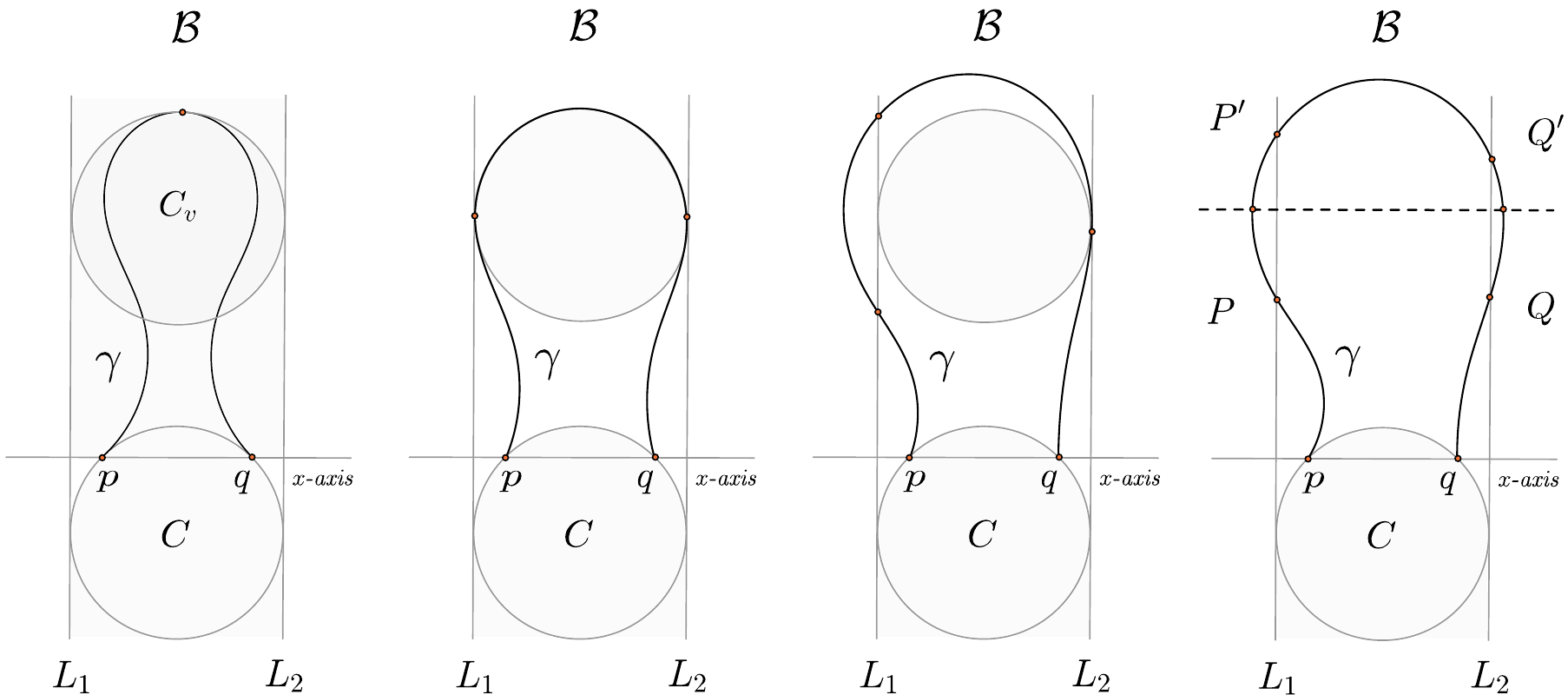}
\caption{An illustration of Corollary \ref{crossect}. Here $p$ and $q$ represent $\gamma(0)$ and ${\gamma(s)}$ respectively. We obtain the fourth illustration from the third one by clockwise rotating the band ${\mathcal B}$ with point of rotation the center of $C$. In this fashion we obtain a pair parallel tangents. The dashed trace at the right corresponds to a cross section for $\gamma$. Refer to \cite{papere} for details.}
\label{figband}
\end{figure}

\begin{corollary} \label{crossect} (See \cite{papere}). Suppose a bounded curvature path $\gamma: I \rightarrow {\mathbb R}^2$ satisfies:
\begin{itemize}
\item $\gamma(0),\gamma(s)$ are points on the $x$-axis.
\item If $C$ is a unit radius circle with centre on the negative $y$-axis, and $\gamma(0),\gamma(s)\in C$, then some point in ${Im}(\gamma)$ lies above $C$.
\end{itemize}
Then $\gamma$ admits parallel tangents and therefore a cross section.
\end{corollary}

\begin{figure}[h]
	\centering
	\includegraphics[width=.7\textwidth,angle=0]{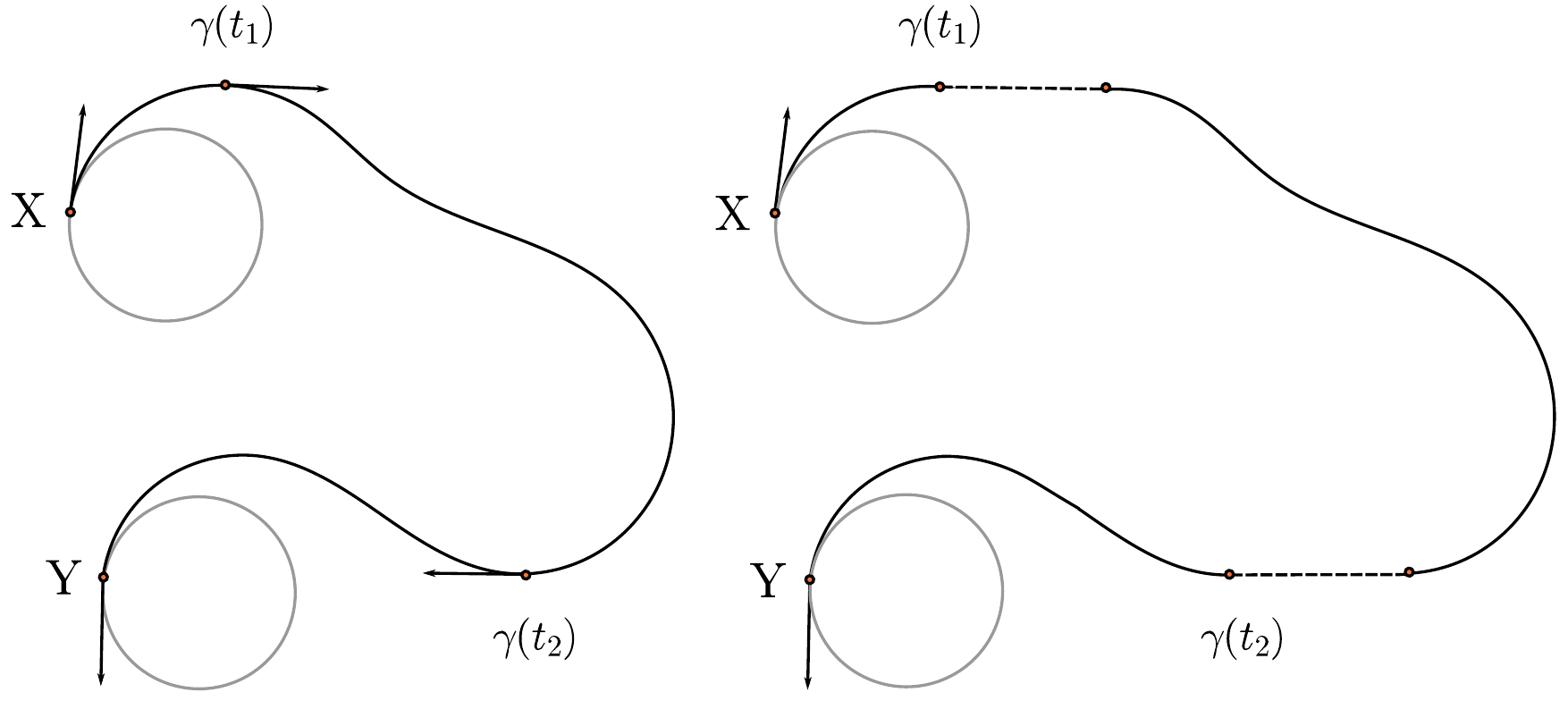}
\caption{Left: An example of a path with parallel tangents. Right:  A path bounded-homotopic to the path at the left.}
 \label{figpartan}
\end{figure}

Next result gives conditions for the existence of parallel tangents.

\begin{proposition}\label{partan} (See \cite{papere}). Bounded curvature paths having parallel tangents are free paths (see Fig. \ref{figpartan}).
\end{proposition}

Dubins in \cite{dubins 2} proved that a bounded curvature path with endpoint condition $\mbox{\sc x,y}\in T\mathbb R^2$ corresponding to an arc of a unit radius circle of length less than $\frac{\pi }{2}$ is an isolated point in $\Gamma(\mbox{\sc x,y})$.  Next we characterize the isolated points in  $\Gamma(\mbox{\sc x,y})$ for all $\mbox{\sc x,y}\in T\mathbb R^2$ satisfying condition D. 
In particular, we prove that bounded curvature paths of length zero are isolated points.


\begin{theorem}\label{c} (Characterization of isolated points). 
\end{theorem}

\begin{itemize}
\item A bounded curvature path from $\mbox{\sc x}\in T\mathbb R^2$ to itself consisting of a single point is an isolated point in $\Gamma(\mbox{\sc x,x})$.
\item A bounded curvature path satisfying $\mbox{\sc x,y}\in T\mathbb R^2$ consisting of a single arc of a unit radius circle of length less than $\pi $ is an isolated point in $\Gamma(\mbox{\sc x,y})$.
\vspace{.2cm}
 \item A bounded curvature path satisfying $\mbox{\sc x,y}\in T\mathbb R^2$ consisting of a concatenation of two arcs of unit radius circles each of length less than $\pi $ is an isolated point in $\Gamma(\mbox{\sc x,y})$.
 \end{itemize}

 \begin{proof}  Let $\gamma \in \Gamma(\mbox{\sc x,x})$ be a loop arbitrarily small. Consider a coordinate system so that $x\in \mathbb R^2$ lies in a radius $r$ circle $C$ while $\gamma$ lying above $C$. By applying Lemma \ref{r1} we conclude that $\gamma$ cannot exist. 
 
 For the second statement. Let $\gamma \in \Gamma(\mbox{\sc x,y})$ be an arc of a unit radius circle of length $\ell<\pi$. Suppose there exists a bounded curvature homotopy ${\mathcal H}_t:[0,1]\rightarrow \Gamma(\mbox{\sc x,y})$ between $\gamma$ and some path $\delta \in \Gamma(\mbox{\sc x,y})$ (different from $\gamma$) such that $\mathcal H_t(0)=\gamma$ and $\mathcal H_t(1)=\delta$. By the continuity of $\mathcal H_t$, there exists $\epsilon>0$ and $\sigma \in \Gamma(\mbox{\sc x,y})$ (different from $\gamma$) such that $\mathcal H_t(\epsilon)=\sigma$ and with the image of $\sigma$ being arbitrarily close to $\gamma$. By applying Lemma \ref{r1} we conclude that $\sigma$ cannot exist. Since $\epsilon>0$ can be chosen to be arbitrarily small, we conclude that $\gamma$ is an isolated point in ${\Gamma}(\mbox{\sc x,y})$.

For the third statement. Let $\gamma \in \Gamma (\mbox{\sc x,y})$ be a concatenation of two arcs of unit radius circles each of length less than $\pi $ say $C_1$ and $C_2$. Suppose there is a bounded curvature homotopy ${\mathcal H}_t:[0,1]\rightarrow \Gamma(\mbox{\sc x,y})$ such that ${\mathcal H}_t(0)=\gamma$ and ${\mathcal H}_t(1)=\delta$ with $\gamma \neq \delta$. Since homotopies are continuous maps we can choose a sufficiently small $\epsilon>0$ such that ${\mathcal H}_t(\epsilon)=\eta$ with the property that $\eta(t')$, for some $t'\in I$, intersects $\gamma$ and is arbitrarily close to the unique intersection (inflection) point between $C_1$ and $C_2$. Since the length of $C_1$ and $C_2$ each is less than $\pi $, by applying Lemma \ref{r1} in between $x$ and $\eta(t')$, or in between $\eta(t')$ and $y$ the result follows. 
\end{proof}

\begin{theorem} \label{singular} A path with a self intersection is free. A free path is bounded-homotopic to a path with a self intersection.

\end{theorem}
\begin{proof} Consider a bounded curvature path having a self intersection at $\gamma(t_1)=\gamma(t_2)$ with $t_1<t_2$. Consider $\gamma(t_1+\delta)$ and $\gamma(t_2-\delta)$ for sufficiently small $\delta>0$. Consider a unit disc containing in its boundary $\gamma(t_1+\delta)$ and $\gamma(t_2-\delta)$. By Corollary \ref{crossect} $\gamma$ must contain a pair of parallel tangents and by Proposition \ref{partan} $\gamma$ is a free path. It is easy to see that a free path is bounded-homotopic to a path with a self intersection by applying the disk sliding procedure in Remark \ref{regdefn} in one of the parallel lines until the deformed piece of the path intersects the other parallel. We leave the detail to the reader.
\end{proof}

\section{Bounded Curvature Paths in $\Omega$}


\begin{definition} Suppose $\gamma: I \rightarrow {\mathbb R}^2$ satisfies condition D. Let,
\begin{itemize}
\item $\Delta(\Omega)\subset \Gamma(\mbox{\sc x,y})$ be the space of embedded bounded curvature paths in $\Omega\subset \mathbb R^2$. 
\item $\Delta'(\Omega)\subset \Gamma(\mbox{\sc x,y})$ be the space of paths bounded-homotopic to paths not in $\Omega\subset \mathbb R^2$.  
\end{itemize}
\end{definition}



We characterize the elements in $\Delta({\Omega})$ by understanding how they intersect  $\partial \Omega$. Our main result, Theorem \ref{mainresultp1} shows that $\Delta(\Omega)\cap \Delta'(\Omega)\neq \emptyset$.

\begin{definition}\label{btpt} Let $\gamma:\mbox{\rm int}(I) \to \Omega\subset \mathbb R^2$ such that its image is not contained in $\partial \Omega$ for every subinterval of $I$. A point $\gamma(t)\in\partial\Omega$ is called {\it boundary tangent point}, see Fig. \ref{figint}. 
\end{definition}

\begin{figure}[h]
	\centering
	\includegraphics[width=1\textwidth,angle=0]{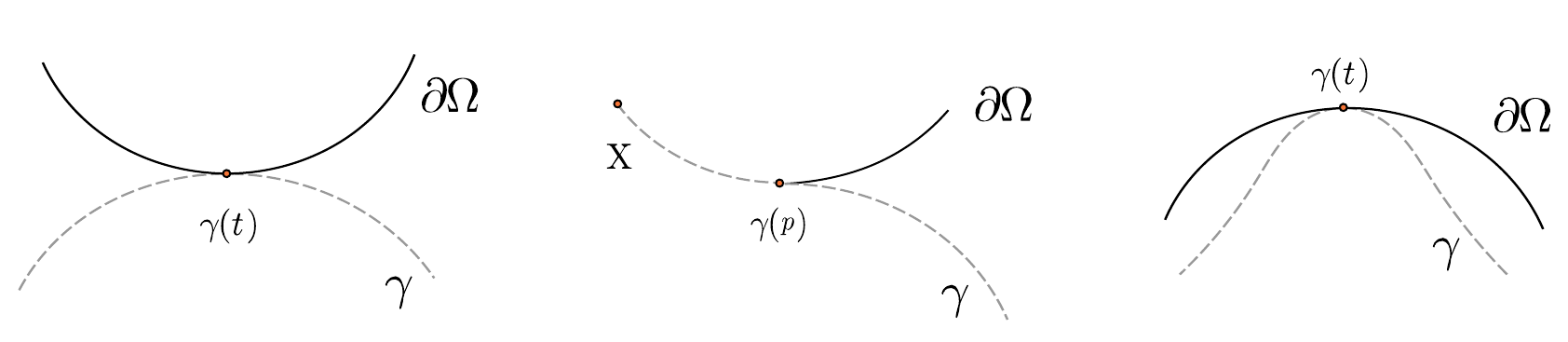}
\caption{The points $\gamma(t)$ at the left and right illustrations are boundary tangent points while $\gamma(p)$ at the center is not a boundary tangent point.}
\label{figint}
\end{figure}

Our first goal is to establish whether embedded bounded curvature paths in $\Omega\subset\mathbb R^2$ have {\it boundary tangent points}. As a first step we prove the following results. Their proofs strongly depend on the bound on curvature and the way these paths intersect $\partial \Omega$.

\begin{lemma}\label{ntp} If a bounded curvature path $\gamma:[0,s] \rightarrow {\mathbb R}^2$ lies in a unit radius disk $\mathcal D$, then either $\gamma$ is entirely in $\partial \mathcal  D$, or the interior of $\gamma$ is disjoint from $\partial {\mathcal D}$.
\end{lemma}

\begin{proof} Suppose there is a bounded curvature path $\gamma:I\to {\mathcal D}$ not entirely in $\partial {\mathcal D}$, such that $\gamma(p)\in \partial {\mathcal D}$, for some $p \in \mbox{int}(I)$. Since $\gamma$ is piecewise $C^2$, there are two cases. Firstly, if  $\gamma''$ is defined (and continuous) in a neighborhood of $p$. Choose coordinates so that $\gamma(p)$ is the origin and $\gamma'(p)$ is the positive $x$-axis with $\mathcal D$ in the upper half plane. Then locally $\partial {\mathcal D}$ has the form $(t,1-\sqrt{1-t^2})$. Write $\gamma$ in the form $(t,f(t))$. Then the curvature bound becomes $\frac{f''}{({1+f'^2})^{\frac{3}{2}}}\leq 1$ using the well known formula for the curvature of a graph. Since $f(0)=0$ and $f'(0)=0$, integrating the differential inequality gives $f(t)\leq 1-\sqrt{1-t^2}$.
But this implies $\gamma$ is disjoint from the interior of $\mathcal D$ near $p$. We conclude $\gamma$ coincides with $\partial {\mathcal D}$ near $p$.
In the second case $\gamma''$ is not defined at $p$. If $\gamma(p)$ is an inflection point then $\gamma$ crosses $\partial {\mathcal D}$ at $\gamma(p)$, contrary to the assumption that $\gamma$ lies entirely in ${\mathcal D}$. If $\gamma(p)$ is not an inflection point we use a similar argument as the one when $\gamma''$ is defined (and continuous) in a neighborhood of $p$ to conclude the proof.
\end{proof}

\begin{corollary}\label{nointcirc} A bounded curvature path having its final position in the interior of either of the disks with boundary $\mbox{\sc C}_l(\mbox{\sc x})$ or $\mbox{\sc C}_r(\mbox{\sc x})$ is free.
\end{corollary}

\begin{proof} Since the distance between the initial and final points $x$ and $y$ satisfies $d(x,y)<2$, by applying Corollary \ref{crossect} the path admits parallel tangents and by Proposition \ref{partan} we conclude that the path is free.
\end{proof}

\begin{theorem} \label{firstreg} \end{theorem}
\begin{itemize}
\item A bounded curvature path in $\Omega \subset \mathbb R^2$ does not have a boundary tangent point in ${\mathcal R}_1$ before first leaving to ${\mathcal R}_2$.
\item A bounded curvature path in $\Omega \subset \mathbb R^2$ does not have a boundary tangent point in ${\mathcal R}_3$ after its last exit from ${\mathcal R}_2$.
\end{itemize}

\begin{proof} Consider the first statement. Let $\gamma$ be a bounded curvature path in $\Omega$. Suppose there exists $t\in \mbox{int}(I)$ such that:
 $$\mbox{Im}(\gamma)\cap \partial \Omega= \{\gamma(t)\}$$
 is a boundary tangent point in ${\mathcal R}_1$ before $\gamma$ reaches ${\mathcal R}_2$.  Change coordinates so that the line through $x$ and $\gamma(t)$ is the new $x$-axis and a new origin $o$ is the midpoint between $x$ and $\gamma(t)$ with the positive $y$-axis passing through ${\mathcal R}_1$. Define ${\mathcal B}$ as in Lemma \ref{r1}. It is easy to check that ${\mathcal R}_1\subset {\mathcal B}$ and so the result follows from Lemma \ref{r1}. An analogous method proves the second statement.
\end{proof}

The first statement in Theorem \ref{firstreg} establishes that no bounded curvature path in $\Omega\subset\mathbb R^2$ have a boundary tangent point in $\partial {\mathcal R}_1\subset \partial \Omega$ before the path enters ${\mathcal R}_2$. Symmetrically, no bounded curvature path in $\Omega$ have a boundary tangent point in $\partial {\mathcal R}_3\subset \partial \Omega$ after leaving ${\mathcal R}_2$ for the final time.

\begin{theorem}\label{r2} Bounded curvature paths in $\Omega \subset \mathbb R^2$ do not have boundary tangent points in ${\mathcal R}_2$.
\end{theorem}

\begin{proof} By considering $\partial {\mathcal R}_2$ as $\partial {\mathcal D}$ in Lemma \ref{ntp} the result follows.
\end{proof}

 \section{The S-Lemma} \label{sectsthm}



The S-Lemma gives a method for characterizing paths in $\Omega\subset \mathbb R^2$ via the turning map and its extremals. Of special relevance will be the existence of {\it maximal inflection points} and their relation with curvature and the diameter of $\Omega$.  

Consider the exponential map $\exp: {\mathbb R} \rightarrow {\mathbb S}^1$.


\begin{definition} For a path $\gamma:I\to{\mathbb R}^2$. The {\it turning map} $\tau$ is defined in the following diagram,
\[ \xymatrix{ I  \ar[d]_{\tau}   \ar@{>}[dr]^{w} &  \\
                     {\mathbb R} \ar[r]_{\exp}  & {\mathbb S}^1} \]

The map $w:I \rightarrow {\mathbb S}^1$ is called the {\it direction map}, and gives the derivative $\gamma'(t)$ of the path $\gamma$ at $t\in I$. The turning map $\tau:I\rightarrow {\mathbb R}$ gives the turning angle the derivative vector makes at $t\in I$ with respect to $\exp(0)$ i.e., the turning angle $\gamma'(t)$ makes with respect to the $x$-axis. \end{definition}

\begin{definition}$\gamma \in \Gamma(\mbox{\sc x,y})$ have a {\it negative direction} if there exists $t\in I$ such that $\langle X, \gamma'(t)\rangle<0$. 
\end{definition}

In order to ensure conditions for the existence of a negative direction we state the following intuitive result whose proof is left to the reader.

\begin{lemma} \label{mono} A $C^1$ path $\gamma(t)=(x(t),y(t))$ with $x:I\rightarrow {\mathbb R}$ not being a monotone function have a negative direction.
\end{lemma}


Consequently, a path having only non-negative directed tangent vectors must have a non-decreasing coordinate function $x:I\rightarrow {\mathbb R}$. Therefore such a path is confined to never travel backwards when projected to the $x$-axis.

\begin{definition}  A boundary tangent point $\gamma(t) \in \partial {\mathcal R}_1$, $t\in I$ is called a returning point if there exists $r\in I$ such that $r<t$ and $\gamma(r)\in {\mathcal R}_2$. A boundary tangent point $\gamma(t) \in \partial{\mathcal R}_3$, $t\in I$ is called a returning point if there exists $r\in I$ such that $t<r$ and $\gamma(r)\in {\mathcal R}_2$.
\end{definition}

Bounded curvature paths in $\Omega\subset \mathbb R^2$ may or may not have a returning point. This depends on the shape of $\Omega$ and therefore on the endpoint in $T\mathbb R^2$. In particular, if $\Omega$ contains a unit radius disk $\mathcal D$ such that $\partial{\mathcal R}_1 \cap {\mathcal D}$ (or $\partial {\mathcal R}_3 \cap {\mathcal D}$) is non-empty then paths with self intersections having returning points can be constructed. However, these paths are not embedded and therefore free by virtue of Theorem \ref{singular}.

\begin{proposition}\label{rpt} A bounded curvature path in $\Omega \subset \mathbb R^2$ having a returning point in $\partial {\mathcal R}_1$ (or $\partial {\mathcal R}_3$) admit a negative direction.
 \end{proposition}

\begin{proof} The existence of a returning point implies non-monotonicity of the first component of $\gamma$ and therefore Lemma \ref{mono} ensures the existence of a point $r\in I$ such that $\langle X, \gamma'(r)\rangle<0$.\end{proof}

\begin{definition}A {\it maximal inflection point} with respect to ${\mbox{\sc x}}\in T{\mathbb R}^2$ is a minimum value of the turning map $\tau: I\rightarrow {\mathbb R}$ (see Fig. \ref{figmaxinf}).
\end{definition}

It is easy to see that if $\gamma$ has a returning point and maximal inflection point at $\gamma(t)$, then $\langle X, \gamma'(t)\rangle<0$.

\begin{definition} Let $\gamma(t)$ be an inflection point, $t \in I$. We denote by $\mathcal I$ the affine line at $\gamma(t)$ spanned by the vector $\gamma'(t)$ and call it the {\it inflection tangent}. The perpendicular line to the inflection tangent at $\gamma(t)$ is denoted by $\mathcal N$ and is called the {\it inflection normal}.
\end{definition}

\begin{remark}\label{inflcoord}Let $\gamma(t)$ be an inflection point. Consider a coordinate system at $\gamma(t)$ with coordinate axes $\mathcal I$ and $\mathcal N$. Such a coordinate system partitions the plane into four quadrants denoted by {\sc I, II, III} and {\sc IV} as usual. We say that the quadrants {\sc I} and {\sc III} and the quadrants {\sc II} and {\sc IV} are opposite quadrants (see Fig. \ref{figsthm}).
\end{remark}

\begin{definition} Suppose that $\gamma$ crosses ${\mathcal N}$ before and after the inflection point $\gamma(t)$. Let $F_1$ be the last time $\gamma$ crosses $\mathcal N$ before reaching $\gamma(t)$ and let $F_2$ be the first time $\gamma$ crosses ${\mathcal N}$ after reaching $\gamma(t)$ (see Fig. \ref{figsthm}).
\end{definition}

\begin{lemma}\label{mipit} (See Fig. \ref{figmaxinf}). Let $\gamma$ be a bounded curvature path with maximal inflection point at $\gamma(t)$ for some $t\in I$.
\begin{itemize}
\item The path $\gamma$ between $F_1$ and $\gamma(t)$ does not cross the inflection tangent $\mathcal I$ in a point other than $\gamma(t)$.
\item The path $\gamma$ between $\gamma(t)$ and $F_2$ does not cross the inflection tangent $\mathcal I$ in a point other than $\gamma(t)$.
\end{itemize}
\end{lemma}
\begin{proof} Suppose that a bounded curvature path $\gamma$ with maximal inflection point $\gamma(t)$ crosses the inflection tangent $\mathcal I$ in between $F_1$ and $\gamma(t)$ (see Fig. \ref{figmaxinf}). In the generic case the path $\gamma$ leaves and then reenters the quadrant {\sc III} obtaining two adjacent intersections say at $\gamma(t_1)$ and $\gamma(t_2)$. Since $\mathcal I$ corresponds to the $x$-axis, we have that the graph of $\gamma$ in between $t_1$ and $t_2$ and in between $t_2$ and $t$ admits minimum and maximum $y$-values respectively. In addition it is easy to see that the graph of $\gamma$ at the minimum is concave and that at the maximum is convex. By virtue of the {\it mountain pass theorem} there exists an inflection point, say $\gamma(t_i)$, other than $\gamma(t)$ in between $F_1$ and $\gamma(t)$, that is, $\tau(t_i)<\tau(t)$ leading to a contradiction. If $\gamma$ intersects $\mathcal I$ in a single point the same method applies. The second assertion is proved using an analogous argument.
\end{proof}

\begin{figure}[h]
	\centering
	\includegraphics[width=1\textwidth,angle=0]{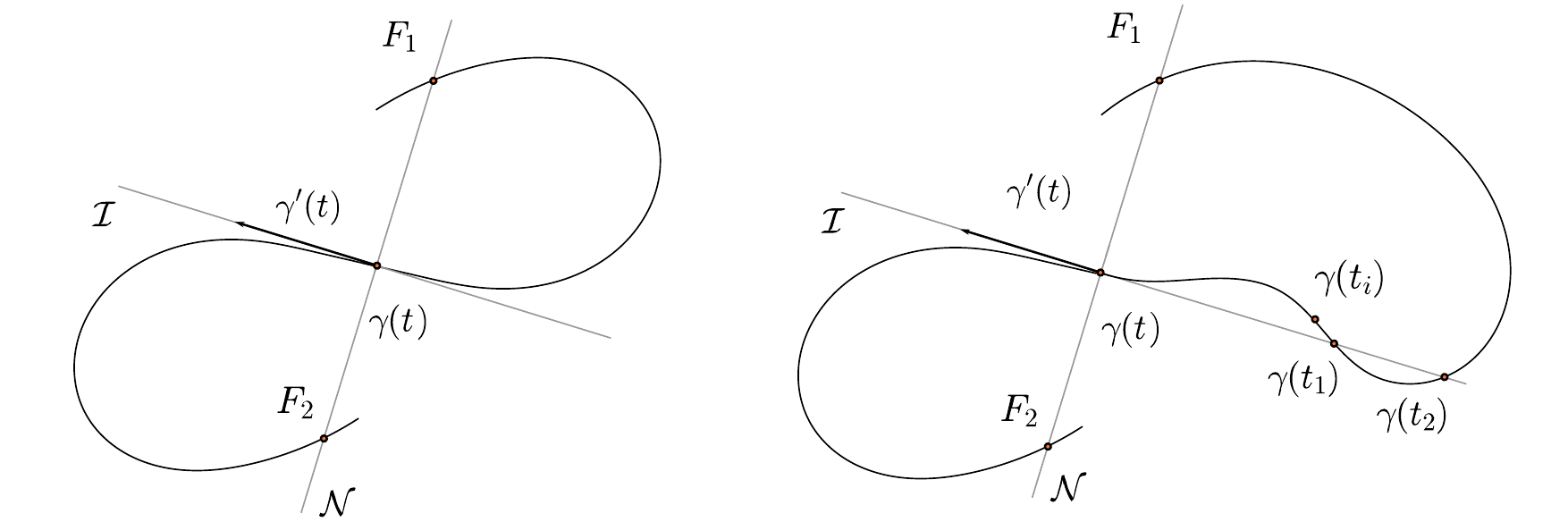}
\caption{The point $\gamma(t)$ at the left is a maximal inflection point but the point $\gamma(t)$ at right is not. }
\label{figmaxinf}
\end{figure}

Recall that $\Delta(\Omega)$ is the space of embedded bounded curvature paths for $\mbox{\sc x,y}\in T\mathbb R^2$ satisfying condition D.

Next we show that paths admitting a returning point have diameter bigger than the diameter of $\Omega\subset \mathbb R^2$ concluding that these are paths not in $\Omega$.

\begin{lemma} \label{sthm} [S-Lemma]. Paths in $\Delta(\Omega)$ do not have returning points.
\end{lemma}

\begin{proof} Suppose there exists an embedded bounded curvature path $\gamma:I\rightarrow {\mathbb R}^2$ in $\Omega\subset \mathbb R^2$ with a returning point at $\partial {\mathcal R}_1$. Since $\theta_1$ and $\theta_2$ have length less than $\pi$; by considering $x=p$ and the returning point to be $q$ in Corollary \ref{crossect} (see Fig. \ref{figband}) we conclude that there exists a pair of parallel tangents and by applying Proposition \ref{partan} we can homotope $\gamma$ until it touches $\partial {\mathcal R}_3$, so the homotoped path has two returning points after applying Proposition \ref{partan}. Let $\gamma(t)$ be a maximal inflection point with coordinate system as in Remark \ref{inflcoord}. By virtue of Lemma \ref{mipit} we have that the path $\gamma$ between $F_1$ and $\gamma(t)$ does not crosses the inflection tangent $\mathcal I$ in a point other than $\gamma(t)$ and symmetrically $\gamma$ also does not crosses $\mathcal I$ in between $\gamma(t)$ and $F_2$. In other words since a maximal inflection point is an extremal of the turning map $\tau$ we have that the two returning points must lie in opposite quadrants.

Our strategy is to establish that none of the trajectories from {\sc x} to {\sc y} are possible for embedded paths in $\Omega$ under the hypothesis of a returning point. The first scenario occurs when $\gamma$ crosses the inflection normal ${\mathcal N}$. In other words we want to analyze the behavior of $\gamma$ as its trajectory travels from the quadrant $ \mbox{\sc IV}$ to $\mbox{\sc III} $ and from the quadrant $ \mbox{\sc I}$ to $\mbox{\sc II}$, see Fig. \ref{figsthm}. We separate this into cases:
\begin{enumerate}
\item $d(F_1,\gamma(t))\geq2 $ and $d(F_2,\gamma(t))\geq 2$.
\item $d(F_1,\gamma(t))<2$ and $d(F_2,\gamma(t))<2$.
\item $d(F_1,\gamma(t))<2$ and $d(F_2,\gamma(t))\geq 2$.
\item $d(F_1,\gamma(t))\geq 2$ and $d(F_2,\gamma(t))< 2$.
\end{enumerate}

The proof of the first case is trivial since it immediately implies that diam$(\gamma)>4$ and therefore by Theorem \ref{diaml4} $\gamma$ is not in $\Omega$.

For the second case we apply Corollary \ref{crossect} to $\gamma$ where the circle $C$ is the right adjacent circle of $\gamma(t)$ and the parallel lines are $L_1={\mathcal I}$ and $L_2$ satisfies the equation $y=2$. We conclude that $\gamma$ must contain a point of the line $L_2$, otherwise it crosses $\mathcal I$ contradicting the hypothesis that $\gamma(t)$ is a maximal inflection point, using Lemma \ref{mipit}.

By an analogous argument applying Corollary \ref{crossect} to $\gamma$ where the circle $C$ is the left adjacent circle of $\gamma(t)$ and parallel lines $L_1={\mathcal I}$ and $L_2$ satisfies equation $y=-2$, we conclude that $\gamma$ must contain a point of the line $L_2$, otherwise it crosses $\mathcal I$ contradicting the hypothesis that $\gamma(t)$ is a maximal inflection point, using Lemma \ref{mipit}. So, we obtain that diam$(\gamma)>4$ therefore $\gamma$ is a path not in $\Omega$. In addition note that the third and fourth cases are identical and also lead to diam$(\gamma)>4$ by a combination of the first and second case.

\begin{figure}[h]
	\centering
	\includegraphics[width=.8\textwidth,angle=0]{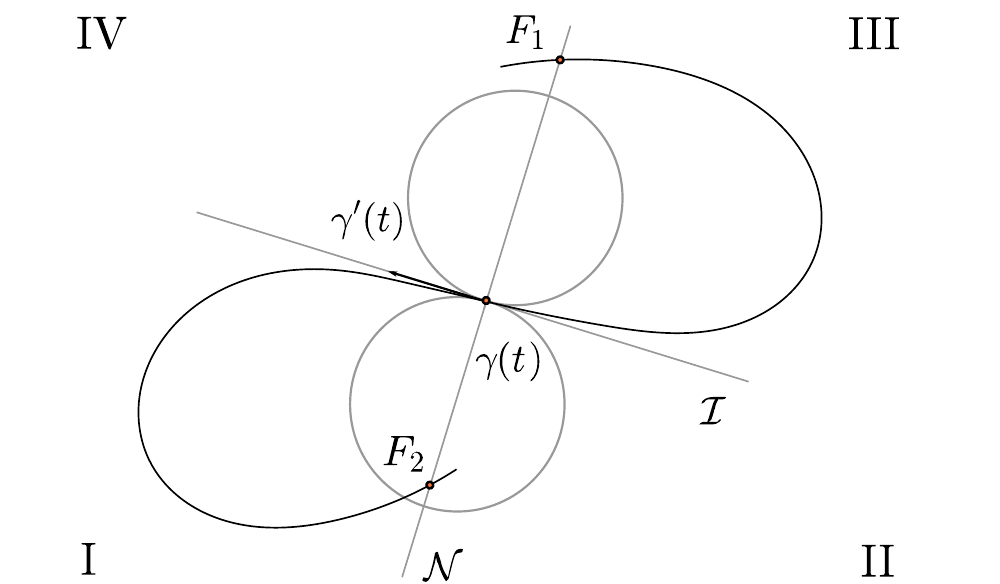}
\caption{A bounded curvature path with diameter bigger than 4.}
 \label{figsthm}
\end{figure}

In conclusion we have excluded the following trajectories for embedded bounded curvature paths in $\Omega$. Here $\rightarrow$ indicates the direction of travel of $\gamma$ where the first and the last element in the sequence contain $x$ and $y$ respectively. 
 
 \begin{itemize}
\item  $\mbox{\sc I} \rightarrow \mbox{\sc IV} \rightarrow \mbox{\sc III} \rightarrow \mbox{\sc I} \rightarrow \mbox{\sc II} $.
\item $\mbox{\sc I} \rightarrow \mbox{\sc IV} \rightarrow \mbox{\sc III} \rightarrow \mbox{\sc I} \rightarrow \mbox{\sc II} \rightarrow \mbox{\sc III}$.
\item $\mbox{\sc IV} \rightarrow \mbox{\sc III} \rightarrow \mbox{\sc I} \rightarrow \mbox{\sc II} $.
\item $ \mbox{\sc IV} \rightarrow \mbox{\sc III} \rightarrow \mbox{\sc I} \rightarrow \mbox{\sc II} \rightarrow \mbox{\sc III}$.
\end{itemize}

Next we establish that no embedded bounded curvature path in $\Omega$, under the hypothesis of a returning point, has initial point $x$ belonging to {\sc III}. Therefore we exclude of the following trajectories:
\begin{itemize}
\item $ \mbox{\sc III} \rightarrow \mbox{\sc I}$.
\item  $ \mbox{\sc III} \rightarrow \mbox{\sc I} \rightarrow \mbox{\sc II}$.
\item  $ \mbox{\sc III} \rightarrow \mbox{\sc I}\rightarrow \mbox{\sc II} \rightarrow \mbox{\sc III}$.
\end{itemize}

Let $\mbox{\sc x,y}\in T{\mathbb R}^2$ such that $x$ lies in {\sc III} and $y$ lies in {\sc I}. Since $\gamma$ admits a returning point, say at $\gamma(t)$ for some $t \in I$, its maximal inflection point is such that $\langle X,\gamma'(t)\rangle < 0$. Now, let $L_1$ be the line satisfying equation $y=2$ and $L_2$ be the line satisfying equation $y=-2$ in the coordinate system considered in Remark \ref{inflcoord}. Consider the following cases:

\begin{itemize}
\item The point $x$ lies in the band between ${\mathcal I}$ and $L_1$ and $y$ lies in the band between ${\mathcal I}$ and $L_2$.

\item The point $x$ lies in the upper half plane with boundary $L_1$ and $y$ lies in the lower half plane with boundary $L_2$.

\item The point $x$ lies in the band between ${\mathcal I}$ and $L_1$ and $y$ lies in the lower half plane with boundary $L_2$.

\item The point $x$ lies in the upper half plane with boundary $L_1$ and $y$ lies in the band between ${\mathcal I}$ and $L_2$.
\end{itemize}

If $x$ belongs to the band in between ${\mathcal I}$ and $L_1$  by virtue of Corollary \ref{crossect} $\gamma$ contains a point in the upper half plane bounded by the line $L_1$. The same argument applies if $y$ belongs to the band between ${\mathcal I}$ and $L_2 $. Therefore we have that diam$(\gamma)>4$ implying that $\gamma$ is a path not in $\Omega$.

The case when $x$ lies in the upper half plane bounded by $L_1$ and $y$ lies in the lower half plane bounded by $L_2$ trivially implies that diam$(\gamma)>4$, therefore $\gamma$ is a path not in $\Omega$. The proof of third and fourth statements involve just a combination of the first and second statements.

The validity of the S-Lemma  for the three remaining configurations for the position of the returning point at $\partial {\mathcal R}_1$ or $\partial {\mathcal R}_3$ is proven using an identical argument as above. Note that the different signs of the curvature at $\gamma(t)$ induce different arrangements for the quadrants.
\end{proof}

\begin{remark} When $\theta_2>\frac{\pi}{2}$, (or $\theta_4>\frac{\pi}{2}$) the returning point may lie on $\partial {\mathcal R}_2$ in an extension of an arc of $\partial {\mathcal R}_1$, ($\partial {\mathcal R}_3$). In any case, Theorem \ref{sthm} also applies.
\end{remark}

\begin{theorem}\label{nt} Paths in $\Delta(\Omega)$ do not have boundary tangent points.
\end{theorem}
\begin{proof} By Theorem \ref{firstreg} bounded curvature paths do not have boundary tangent points with $\partial {\mathcal R}_1$ before entering ${\mathcal R}_2$, and in $\partial{\mathcal R}_3$ after leaving $ {\mathcal R}_2$ for the final time. By Lemma \ref{sthm}, embedded bounded curvature paths do not have boundary tangent points at $\partial {\mathcal R}_1$. By a symmetrical argument, embedded bounded curvature paths do not have boundary tangent points in $\partial {\mathcal R}_3$. By Theorem \ref{r2} bounded curvature paths do not have boundary tangent points with $\partial {\mathcal R}_2$. Since $\partial \Omega$ is formed by $\partial {\mathcal R}_1$, $\partial{\mathcal R}_2$ and  $\partial{\mathcal R}_3$, the result follows.
\end{proof}

\begin{theorem}\label{insbound} Paths in $\Delta(\Omega)$ have bounded length.
\end{theorem}
\begin{proof}Let $\gamma \in \Delta(\Omega)$ be a path of arbitrarily large length. Suppose $\gamma$ does not have a negative direction. Since $\Omega\subset{\mathbb R}^2$ is bounded and $\gamma$ is a path of arbitrary large length that only travels forward, we have that $\gamma$ must leave $\Omega$ leading to a contradiction. Suppose $\gamma$ have a negative direction, since the turning map $\tau:I \to {\mathbb R}$ is a continuous function defined on a compact domain, $\tau$ admits a maximal inflection point. As a consequence of Lemma \ref{mipit} the trajectory of $\gamma$ after the maximal inflection point (see Fig. \ref{boundlength}):

\begin{itemize}
\item Lies in between the lines $\mathcal I$ and $L_2$ traveling forward.
\item Turns back and lies in between the lines $L_1$ and $L_2$.
\item Crosses $L_2$.
\item Turns back and crosses $L_1$ twice.
\end{itemize}

\begin{figure}[h]
	\centering
	\includegraphics[width=.7\textwidth,angle=0]{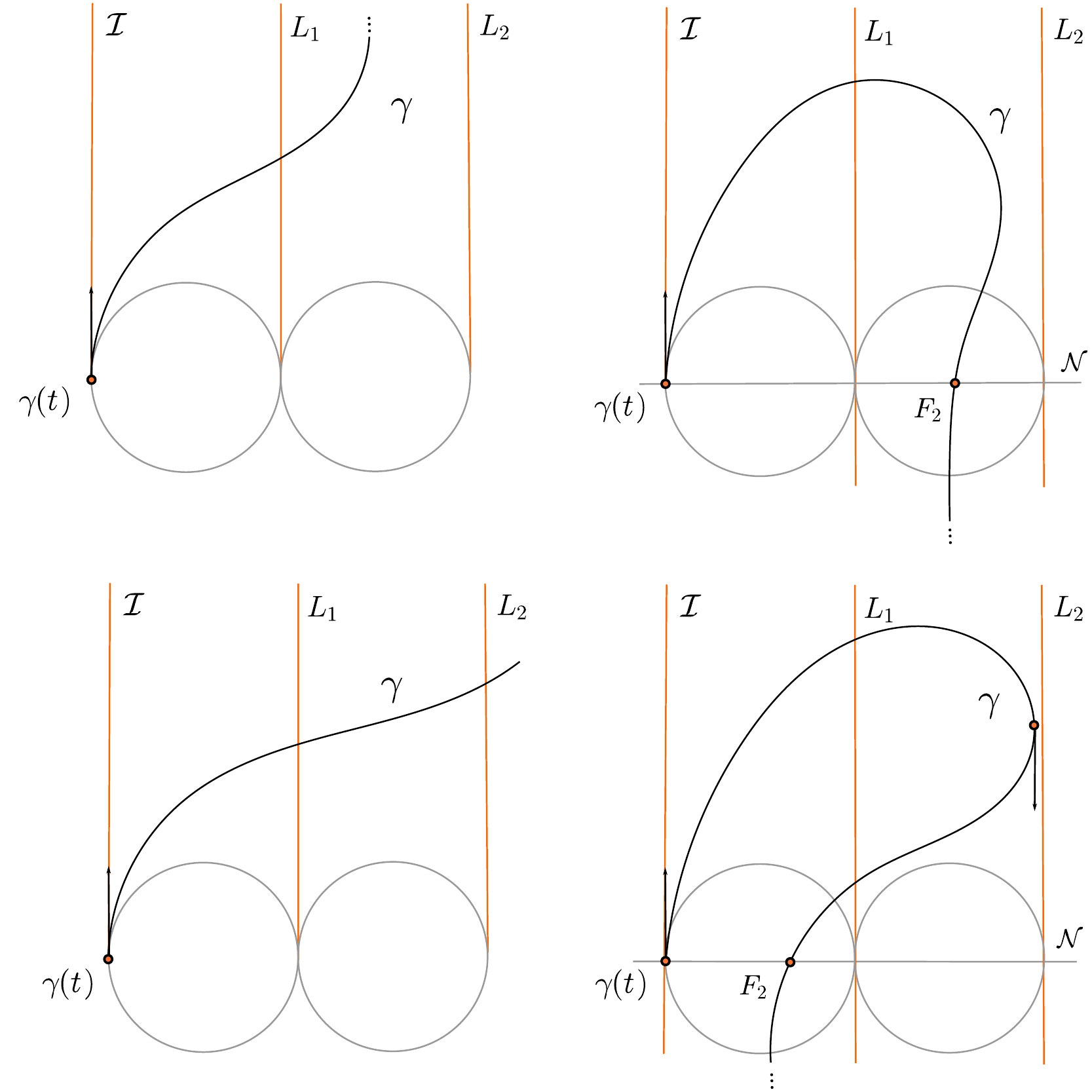}
\caption{Possible trajectories for $\gamma$ after the maximal inflection point $\gamma(t)$. By Lemma \ref{mipit} (after the maximal inflection point) the path $\gamma$ may cross the inflection tangent $\mathcal I$ before intersecting $\mathcal N$ (see Fig. \ref{figmaxinf}).}
 \label{boundlength}
\end{figure}

\noindent Here, the inflection tangent $\mathcal I$ and $L_1$, $L_2$ are parallel lines tangent to unit radius circles. Since the length of $\gamma$ is chosen to be arbitrarily large, it is easy to see that under the possible trajectories of $\gamma$ after the maximal inflection point a diameter bigger than 4 is always achieved. The first two cases run with an analogous argument as applied to paths that only travel forward (see Fig. \ref{boundlength} top). The third case immediately implies $\mbox{\rm diam}(\gamma)>4$, the last case implies the existence of parallel tangents, by Theorem \ref{partan} $\gamma$ is a free path and by Theorem \ref{nt} these are paths not in $\Omega$ (see Fig. \ref{boundlength} bottom). Since the possible unbounded length paths always have diameter bigger than 4, by Theorem \ref{diaml4} the result follows.
\end{proof}


\begin{corollary}\label{cannotsing} Paths in $\Delta(\Omega)$ are not bounded-homotopic to paths with a self intersection.
\end{corollary}
\begin{proof}  Consider $\gamma\in \Delta(\Omega)$ and a homotopy of bounded curvature paths $\mathcal H_t$ such that ${\mathcal H}_t(0)=\gamma$, and ${\mathcal H}_t(1)$ is a path in with a self intersection. Then there exists $r\in I$ such that ${\mathcal H}_t(r)=\sigma$ is the first path admitting a self intersection in ${\mathcal H}_t$. By Theorem \ref{singular} we have that $\sigma$ is free and therefore bounded-homotopic to a path of arbitrary large length contradicting Theorem \ref{insbound}.\end{proof}


\section{Non-uniqueness of the homotopy class of bounded curvature paths}\label{secnotriv}

Next we present our main result.

\begin{theorem}\label{mainresultp1} If $\mbox{\sc x,y}\in T{\mathbb R}^2$ satisfies condition D. Then, $\Gamma(\mbox{\sc x,y)}$ has least two homotopy clasess being these ${\Delta}(\Omega)$ and ${\Delta}'(\Omega)$. In particular the elements in ${\Delta}(\Omega)$ are not free.  
\end{theorem}

\begin{proof} If $\Gamma(\mbox{\sc x,y)}$ contains an isolated point the result immediately follows. Consider $\gamma \in {\Delta}(\Omega)$. Suppose there exists $\delta \in {\Delta}'(\Omega)$ together with a bounded curvature homotopy $\mathcal H_t: [0,1] \rightarrow \Gamma(\mbox{\sc x,y)}$ such that ${\mathcal H}_t(0)=\gamma$ and ${\mathcal H}_t(1)=\delta$. Then there exists $p \in I$ such that ${\mathcal H}_t(p)$ has a boundary tangent point contradicting via Theorem \ref{nt} the continuity of $\mathcal H_t$. The elements in ${\Delta}(\Omega)$ are not free by Theorem \ref{insbound}. 
\end{proof}

We conclude that (embedded) paths in $\Delta(\Omega)$ are not bounded-homotopic to paths in $\Delta'(\Omega)$. In particular, the spaces $\Gamma(\mbox{\sc x,y})$ with $\mbox{\sc x,y}\in T{\mathbb R}^2$ satisfying condition {\sc D} are not path connected implying the existence of at least two different homotopy classes in $\Gamma(\mbox{\sc x,y})$. We conjecture that $\Delta(\Omega)$ is an isotopy class in $\Gamma(\mbox{\sc x,y})$. 

\section{An Application to Motion Planning }\label{ext}

In the decline design of an underground mine, the mine is considered as a 3-dimensional network (strategic locations in the mine are represented as directed nodes with links establishing connections between such locations) see Fig. \ref{figmine}. Additional restrictions on the structure of the 3-dimensional network are given by navigability restriction on turning radius for vehicles, and the inclination (or gradient) for the ramps.

Sussmann in 1995 answered the problem of finding minimal length bounded curvature paths in $\mathbb R^3$ \cite{sus3}. The gradient constraint is not taken into consideration being the important result in \cite{sus3} inviable for applications to motion planning of robots in 3-space due to mechanical limitations. 

The approach of minimizing the cost of the links corresponds to considering the projected problem in the horizontal plane. A planar path (length minimiser) can be lifted into the 3-space while keeping a uniform gradient. The lifted path will satisfy the gradient constraint if and only if  the length of the planar path reaches a lower bound dependent on the vertical displacement between the end points of the link. If the length of the minimum length path is less than the given lower bound we can attempt to extend (homotope) the path to reach the required length. If the projected nodes $\mbox{\sc x,y}\in T{\mathbb R}^2$ are such that a region $\Omega$ is obtained then by virtue of Theorem \ref{nt} we have that such paths are trapped in $\Omega$. Therefore, the desired lower bound for the length of the paths between the projected nodes may not be achieved. 

\begin{figure}[h]
	\centering
	\includegraphics[width=.7\textwidth,angle=0]{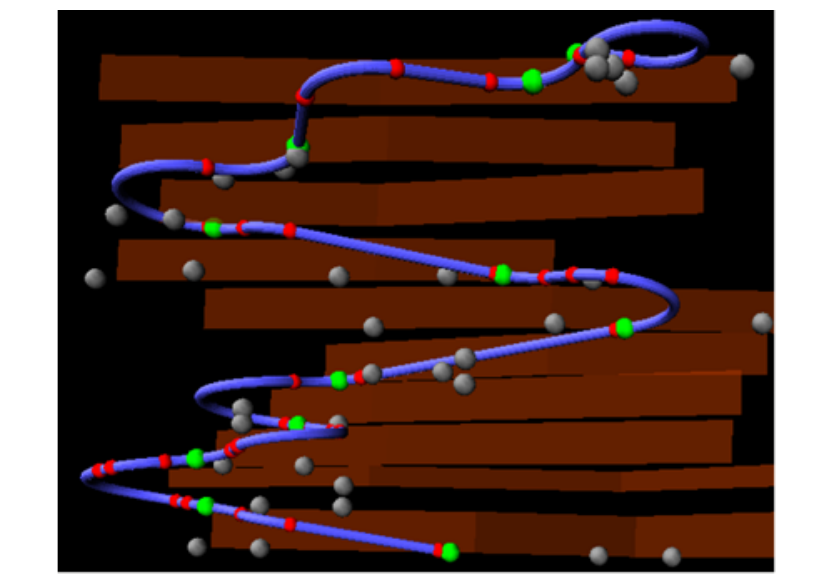}
\caption{An underground mine design obtained by DOT.}
\label{figmine}
\end{figure}

As an immediate consequence of Theorem \ref{mainresultp1} and in the context of the problem previously described we have the following result.

\begin{corollary}\label{noarblength}Suppose that $\mbox{\sc x,y}\in T{\mathbb R}^2$ satisfies condition {\sc D}. Then the minimal length element in $\Gamma(\mbox{\sc x,y})$ is not bounded-homotopic to a free path.
 \end{corollary}

As the result of the research conducted at The University of Melbourne research group in underground mine optimization, the software DOT (Decline Optimisation Tool) has been developed \cite{dot}. The algorithm in DOT assumes the validity of Corollary \ref{noarblength}. The algorithm searches for the second shortest Dubins path between the given endpoint condition. In this fashion the algorithm in DOT becomes optimal. Fig. \ref{figmine} illustrates an underground mine design obtained by DOT.


\bibliographystyle{amsplain}

\end{document}